\documentclass[1p,10pt]{elsarticle}

\makeatletter
\def\ps@pprintTitle{%
   \let\@oddhead\@empty
   \let\@evenhead\@empty
   \let\@oddfoot\@empty
   \let\@evenfoot\@oddfoot
}
\makeatother

\usepackage{amsfonts,amscd,amsthm}
\usepackage{mathtools}
\usepackage{xspace}
\usepackage{booktabs}
\usepackage{array}
\usepackage{siunitx}
\usepackage{hyperref}
\usepackage{cleveref}
\usepackage{tikz}
\usetikzlibrary{cd, patterns, plotmarks}
\usepackage{xparse}
\usepackage{subcaption}
\usepackage{algorithm}
\usepackage{algpseudocodex}
\usepackage{rotating}
\usepackage{float}
\usepackage{pgfplots}
\usepackage{csvsimple}
\usepackage{ifthen}
\usepackage{makecell}
\pgfplotsset{compat=1.18}

\crefname{thm}{Theorem}{Theorems}
\crefname{lem}{Lemma}{Lemmas}
\crefname{defn}{Definition}{Definitions}
\crefname{figure}{Fig.}{Figs.}
\crefname{table}{Table}{Tables}
\crefname{algorithm}{Algorithm}{Algorithms}
\crefname{subsection}{Section}{Sections}

\newtheorem{thm}{Theorem}[section]

\newtheorem{lem}{Lemma}[section]
\newtheorem{prop}{Proposition}[section]

\newtheorem{defn}{Definition}[section]
\newtheorem{rem}{Remark}[section]
\newtheorem{assum}{Assumption}[section]
 
\newcommand{\wlg}{w.l.o.g\xspace}
\newcommand{\lchain}{L-chain\xspace}
\newcommand{\lchains}{L-chains\xspace}
\newcommand{\nlintersec}{minimal \((\level+1)\)-intersection\xspace}
\newcommand{\dofs}{DoFs\xspace}
\newcommand{\var}{\texttt}

\newcommand{\RR}{\mathbb{R}}
\newcommand{\NN}{\mathbb{N}}

\newcommand{\mbf}[1]{\boldsymbol{#1}}
\newcommand{\image}{\mathrm{Im}}
\newcommand{\kernel}{\mathrm{Ker}}
\newcommand{\cohomo}{\mathrm{H}}
\newcommand{\complex}[1]{\mathfrak{#1}}

\newcommand{\domain}{\Omega}
\newcommand{\sideconfig}{S}

\newcommand{\level}{\ell}

\newcommand{\degree}{p}
\newcommand{\bspan}[1]{\spanop(#1)}
\newcommand{\meshsymb}{Q}
\newcommand{\meshel}{\meshsymb}
\newcommand{\mesh}{\mathcal{\meshsymb}}
\NewDocumentCommand\interbox{mg}{\square_{#1 \IfNoValueF{#2}{,#2}}}
\newcommand{\knotvec}{\Xi}
\newcommand{\knot}{\xi}

\newcommand{\bsp}{B}
\newcommand{\hb}{H}
\newcommand{\thb}{T}
\newcommand{\funcspace}[1]{\mathbb{#1}}
\newcommand{\basis}[1]{\mathcal{#1}}
\newcommand{\bspbasis}{\basis{\bsp}}
\newcommand{\hbbasis}{\basis{\hb}}
\newcommand{\thbbasis}{\basis{\thb}}
\newcommand{\bspspace}{\funcspace{\bsp}}
\newcommand{\hbspace}{\funcspace{\hb}}

\newcommand{\xbsp}{\beta}

\newcommand{\scltest}{\tau}
\newcommand{\scltrial}{\sigma}
\newcommand{\vectest}{\boldsymbol{v}}
\newcommand{\vectrial}{\boldsymbol{u}}
\newcommand{\forcing}{\boldsymbol{f}}
\newcommand{\solution}{\boldsymbol{u}}

\newcommand{\Hone}{H^{1}(\domain)}
\newcommand{\Hcurl}{\boldvec{H}(\curl;\domain)}
\newcommand{\Ltwo}{L^2(\domain)}
\newcommand{\Hcurlzero}{\boldvec{H}_0(\curl;\domain)}

\newcommand{\boldvec}[1]{\boldsymbol{\mathrm{#1}}}
\newcommand{\abs}[1]{\lvert #1 \rvert}

\newcommand{\Bll}{\bspbasis_{\level, \level+1}^{\boldvec{0}}}

\newcommand{\chain}[2]{C_{#1, #2}}

\DeclareMathOperator{\supp}{supp}
\DeclareMathOperator{\mot}{mot}
\DeclareMathOperator{\trunc}{trunc}

\DeclareMathOperator{\closedsupp}{\overline{supp}}

\DeclareMathOperator{\grad}{\boldvec{grad}}
\DeclareMathOperator{\curl}{curl}
\DeclareMathOperator{\spanop}{span}
 
\begin{document}
\begin{frontmatter}
	\title{
		Construction of exact refinements for the two-dimensional hierarchical B-spline de
		Rham complex
	}
	\author[tud]{Diogo C. Cabanas\corref{cor}}
	\ead{d.costacabanas@tudelft.nl}
	\author[byu]{Kendrick M. Shepherd}
	\ead{kendrick\_shepherd@byu.edu}
	\author[tud]{Deepesh Toshniwal}
	\ead{d.toshniwal@tudelft.nl}
	\author[usc]{Rafael Vázquez}
	\ead{rafael.vazquez@usc.es}

	\address[tud]{
		Delft Institute of Applied Mathematics, Delft University of Technology, The Netherlands
	}
	\address[byu]{
		Civil \& Construction Engineering, Brigham Young University, Provo UT 84602, USA
	}
	\address[usc]{
		Department of Applied Mathematics, University of Santiago de Compostela, and Galician
		Centre for Mathematical Research and Technology (CITMAga), Santiago de Compostela, Spain
	}
	\cortext[cor]{Corresponding Author}
\begin{abstract}
The de Rham complex arises naturally when studying problems in electromagnetism and fluid
mechanics.
Stable numerical methods to solve these problems can be obtained by using a discrete de Rham
complex that preserves the structure of the continuous one.
This property is not necessarily guaranteed when the discrete function spaces are
hierarchical B-splines, and research shows that an arbitrary choice of refinement domains
may give rise to spurious harmonic fields that ruin the accuracy of the solution.
We will focus on the two-dimensional de Rham complex over the unit square \(\Omega \subseteq
\mathbb{R}^2\), and provide theoretical results and a constructive algorithm to ensure that
the structure of the complex is preserved: when a pair of functions are in conflict some
additional functions, forming an L-chain between the pair, are also refined.
Another crucial aspect to consider in the hierarchical setting is the notion of
admissibility, as it is possible to obtain optimal convergence rates of numerical solutions
and improved stability by limiting the multi-level interaction of basis functions.
We show that, under a common restriction, the admissibility class of the first space of the
discrete complex persists throughout the remaining spaces.
As such, admissible refinement can be combined with our new algorithm to obtain admissible
meshes that also respect the structure of the de Rham complex.
Moreover, we detail how our algorithm can be easily included in standard adaptive mesh
refinement schemes.
Finally, we include numerical results that motivate the importance of the previous concerns
for the vector Laplace and Maxwell eigenvalue problems.
\end{abstract}
\end{frontmatter}

\section{Introduction}

Over the last few years, there have been several developments in numerical methods designed
to approximate the solution of partial differential equations (PDEs). Finite element
exterior calculus (FEEC) \cite{Arnold2006, Arnold2009} is one of these developments, which
has gained a lot of traction since its earlier publications \cite{Arnold2002} due to its
ability to describe how stable numerical methods can be achieved for a broad class of PDEs.
Key elements used to ensure the stability of discrete methods are the topological and
differential-geometric structure of PDEs, which are captured in FEEC through complexes and,
more specifically, their cohomology classes. By choosing an appropriate finite element
space, it is possible not only to approximate the continuous solution, but also to guarantee
that the method is stable --- provided that the continuous and discrete complexes are
cohomologically equivalent.
An example of such a complex, which we approach in this work, is the \emph{de Rham} complex.
Several PDEs can be understood, and successfully discretized, in the context of the de Rham
complex, such as those arising in  Stokes flow, Navier-Stokes, Maxwell's equation,
Vector/Scalar Laplacians, among others \cite{Arnold2006}.

Another important breakthrough in the field of numerical analysis was the introduction of
spline functions \cite{Boor2001, Schumaker2007} as a compelling alternative to the more
classical finite element method \cite{Hughes2000, Ciarlet2002, Brenner2008}. The former
approach was given the name of isogeometric analysis (IGA), and it is explained in great
detail in the seminal paper by Hughes et al. \cite{Hughes2005}. The motivation behind this
method was to create a better synergy between numerical-analysis and industrial-design
teams, which have been using splines in CAD-based modelling for a long time
\cite{Farin2006}. One of the benefits of the higher smoothness allowed in splines is that
they provide better approximation properties per degree of freedom (DoF) than standard
\(C^0\) or \(C^{-1}\) finite element spaces \cite{Evans2009, Bressan2019, Sande2019,
Sande2020}. Another upside in IGA is the flexibility in defining new spline spaces from
basic splines, and in particular locally refined spline spaces. These allow us to increase
the density of \dofs near relevant local features of the solution being computed, increasing
the accuracy in those regions and reducing the computational cost. Several types of splines
have been explored in the IGA literature, such as T-splines \cite{Sederberg2003,
Sederberg2004}, hierarchical B-splines \cite{Vuong2011, Buffa2015, D'Angella2018} and
truncated hierarchical B-splines \cite{Giannelli2012, Giannelli2013}, and the more recent
polar splines \cite{Speleers2021}.
In this work, our attention will be centered around hierarchical B-splines and their
truncated counterpart.
The general idea of hierarchical basis is to take a sequence of function space basis, each
corresponding to a level providing more DoFs than the former, and then applying a
selection mechanism to determine which basis functions from each level are present in the
final hierarchical basis \cite{Kraft1998}.

When working with a hierarchical B-spline complex in \(\RR^n, n\geq1\), it is not true that
this will generate an exact subcomplex of the continuous one for an arbitrary refinement
\cite{Evans2018, Shepherd2024}.
This is due to the fact that a poor choice of the refinement pattern can remove basis
functions from a coarser level and add basis functions from a finer level such that they
describe function spaces with distinct topological properties.
This leads to the introduction of spurious harmonic forms, which can easily ruin the
accuracy and stability of the numerical method being used.

Focusing on the two-dimensional case, the first contribution of this work is a constructive
method to modify the refinement pattern, and to recover an exact complex for spaces of
hierarchical B-splines, by adding some elements to refine. Leveraging the theoretical
results from \cite{Shepherd2024}, the modified refinement works as follows: when a pair of
functions supported in the refined domain is \textit{problematic}, following the definition
in \cite{Shepherd2024}, we also refine the support of some additional functions, forming an
L-chain between the pair.
The L-chain is simply a path in taxicab geometry where two points, in our case basis
function indices, are connected through successive exhaustion of their difference in each
dimension. 
By this change, we can topologically connect degrees of freedom in such a way that
cohomology is preserved between levels.
We present theoretical results to show that the shape of the L-chain reduces the amount of
recursive checks for problematic pairs compared to other possible refinement paths.
We also provide algorithms to easily implement the refinement scheme in any existing IGA
code.

Another important consideration when discretizing with hierarchical B-spline spaces is the
concept of admissibility, which limits the interaction of functions from very different
refinement levels \cite{Buffa2015}. Admissibility is a required property to obtain optimal
convergence rates in the theoretical study of adaptive methods \cite{Buffa2017, Buffa2022}.
Additionally, it has been shown on time-dependent problems that admissibility improves the
stability of the method when applying adaptive coarsening \cite{Carraturo2019}. The second
contribution of this work is to show that admissibility for the first space of the complex
implies the same property, with the same admissibility class, for all the spaces in the
complex. This result is valid for both hierarchical and truncated hierarchical spline bases,
and requires a common assumption in the IGA literature: the refined regions must be formed
by the union of supports of functions from the previous level. Furthermore, our new
algorithm can be easily combined with admissible refinement, to obtain an exact complex such
that all the spaces respect the admissibility property.

The structure of the paper is as follows: in \Cref{sec:preliminaries} we give a basic
overview of the de Rham complex and its discretization. Next, in \Cref{sec:spline-spaces} we
describe the construction of hierarchical B-splines and the corresponding hierarchical
B-spline de Rham complex. In \Cref{sec:admissibility,sec:exact-meshes}, we present the main
theoretical results for admissibility and exactness, respectively. These are followed by
\Cref{sec:algorithms} with their accompanying algorithms. We test our approach numerically
in \Cref{sec:numerical-results}, where we showcase various examples using both the vector
Laplace problem and Maxwell's eigenvalue problem to illustrate the use of the method as an
adaptive scheme. Finally, in \Cref{sec:conclusion} we give some concluding remarks.
 \section{Preliminaries}\label{sec:preliminaries}

We start by introducing an overview of the basic objects and concepts used in FEEC. We
specialize this overview for the two-dimensional case, which is our primary focus. Moreover,
we simplify the exposition by avoiding the use of differential forms and instead use vector
calculus proxies. For a more general discussion, the reader is invited to consult
\cite{Arnold2018}.

\subsection{The Hilbert de Rham complex}
Let $\domain \subset \RR^2$ be a bounded domain with a piecewise-smooth Lipschitz boundary.
The following sequence, $\complex{R}$, of Hilbert spaces and connecting differential
operators is called the Hilbert de Rham complex,
\begin{equation*}
	\begin{tikzcd}
		\complex{R}~:~0\arrow{r}{} & \mathbb{R} \arrow{r}{\iota} & \Hone \arrow{r}{\grad} & \Hcurl \arrow{r}{\curl} & \Ltwo\arrow{r}{} &0\;,
	\end{tikzcd}
\end{equation*}
where $\iota$ is the inclusion operator, $\grad$ is the gradient operator defined in
coordinates by \(\grad u = (\partial_x u, \partial_y u)\), \(\curl\) is the scalar curl
operator in two dimensions and is defined in coordinates as $\curl \mbf{v} := \partial_x v_2
	- \partial_y v_1$, $\Ltwo$ is the space of square-integrable functions on $\domain$, and
\begin{align*}
	\Hone  & := \left\{ u \in \Ltwo : \grad u \in [\Ltwo]^2 \right\}\;,                           \\
	\Hcurl & := \left\{ \boldsymbol{v} \in [\Ltwo]^2 : \curl \boldsymbol{v} \in \Ltwo \right\}\;.
\end{align*}
We will refer to these spaces as \(\funcspace{X}^0 = \Hone, \funcspace{X}^1=\Hcurl\) and
\(\funcspace{X}^2=\Ltwo\). Similarly, in this paper, we will particularly work with the
Hilbert de Rham complex with homogeneous boundary conditions, namely
\begin{equation*}
	\begin{tikzcd}
		\complex{R}_0~:~0\arrow{r}{} &  H^1_0(\domain) \arrow{r}{\grad} & \boldvec{H}_0(\curl;\domain) \arrow{r}{\curl} & \Ltwo\arrow{r}{\int} &\mathbb{R} \arrow{r}{} &0\;,
	\end{tikzcd}
\end{equation*}
where $H^1_0(\domain)$ and $\boldvec{H}_0(\curl;\domain)$ correspond to $\Hone$ and
$\Hcurl,$ respectively, but with vanishing trace, meaning null boundary
conditions of each space:
\begin{align*}
	H^1_{0}\left(\domain\right) &\coloneq \left\{u \in H^1\left(\domain\right):
	u|_{\partial\domain} = 0\right\}, \\
		\boldvec{H}_0\left(\curl;\domain\right) &\coloneq \left\{\boldsymbol{v} \in
		\boldvec{H}\left(\curl;\domain\right): \boldsymbol{v}|_{\partial\domain}\times
		\boldvec{n} = 0\right\},
\end{align*}
where \(\partial \domain\) is the boundary of \(\domain\) and \(\boldvec{n}\) is the outward
unit normal.

The sequences $\complex{R}$ and $\complex{R}_0$ are called complexes because of two
properties. First, each differential operator maps from one space into the next space,
meaning that, for the $\complex{R}$ and $\complex{R}_0$ sequences it respectively holds
\begin{gather*}
	\image(\grad) \subset \Hcurl\;,\qquad
	\image(\curl) \subset \Ltwo\;,\\
\image(\grad) \subset \boldvec{H}_0(\curl;\domain)\;,\qquad
	\image(\curl) \subset \Ltwo\;,
\end{gather*}
where \(\image\) denotes the image of a given map; similarly, we will use \(\kernel\) for
the kernel. 
Second, the composition of the two operators is identically zero, meaning for example
$\curl (\grad f) = 0$ for any $f \in \Hone$.
Given these two properties, we can define the following (quotient) spaces called the
\(k\)-cohomology spaces of $\complex{R}$, and $\complex{R}_0$, for \(k=0,1,2\),
\begin{gather*}
	\cohomo^0(\complex{R}) := \kernel(\grad)/\mathbb{R},\
	\cohomo^1(\complex{R}) := \kernel(\curl)/\image(\grad),\
	\cohomo^2(\complex{R}) := \Ltwo/\image(\curl),\\
\cohomo^0(\complex{R}_0) := \kernel(\grad),\
	\cohomo^1(\complex{R}_0) := \kernel(\curl)/\image(\grad),\
	\cohomo^2(\complex{R}_0) := \Ltwo/\image(\curl),
\end{gather*}
The elements of $\cohomo^k(\complex{R})$ are called harmonic scalar/vector fields, and the
elements of $\cohomo^k(\complex{R}_0)$ are called harmonic scalar/vector fields with zero
trace.
These cohomology spaces are important because, first, they capture the topological
properties of the domain $\domain$ and, second, the solutions to scalar and vector
Laplacians are well-defined up to harmonic fields. The following two paragraphs elaborate on
this structure that underlies the Hilbert de Rham complex.

Due to the Universal Coefficient Theorem \cite[Theorem~3.2]{Hatcher}, there is a one-to-one
correspondence between the cohomology spaces and the so-called Betti numbers of $\domain$,
$\beta_j(\domain)$: the vector space dimension of $\cohomo^j(\complex{R})$ is equal to the
number of $j$-dimensional holes in $\domain$. That is, $\dim
	\cohomo^0(\complex{R})=\beta_0(\domain)$ is the number of connected components of $\domain$,
$\dim \cohomo^1(\complex{R})=\beta_1(\domain)$ is the number of holes in $\domain$, and
$\dim \cohomo^2(\complex{R})=\beta_2(\domain)=0$ is the number of voids (none in the planar
setting since there are no two-dimensional holes in $\domain$). For the space of
differential forms with homogeneous boundary conditions, due to Lefschetz Duality
\cite[Theorem~3.43]{Hatcher}, this above-mentioned correspondence between dimension of
cohomology spaces and Betti numbers is flipped in that $\dim
	\cohomo^0(\complex{R}_0)=\beta_2(\domain), \dim \cohomo^1(\complex{R}_0)=\beta_1(\domain),$
and $\dim \cohomo^2(\complex{R})=\beta_0(\domain).$ In particular, in this work we will
focus only on simply-connected domains $\domain$ with homogeneous boundary conditions, which
implies that $\cohomo^0(\complex{R}_0) = \emptyset = \cohomo^1(\complex{R}_0)$ and
$\cohomo^2(\complex{R}_0) \cong \RR$.

The second reason why cohomology spaces are important is that harmonic fields in
$\cohomo^1(\complex{R}_0)$ are the solutions to the homogeneous vector Laplace problem,
$\Delta \mbf{v} = 0$, and the harmonic fields in $\cohomo^0(\complex{R}_0)$ are the
solutions to the homogeneous scalar Laplace problem, $\Delta u = 0$, with homogeneous
essential boundary conditions. Given the assumption of simply-connectedness for $\domain$,
this means that both problems admit only the trivial solution.

\subsection{Structure-preserving discretizations}

The core principle of FEEC is to preserve the structural connections mentioned above in the
discrete setting. This can be achieved by choosing appropriate finite element spaces for
approximating scalar and vector fields. Failure to preserve the cohomological structure in
the discrete setting can cause, in particular, spurious harmonic fields to appear in the
discrete setting, which can lead to highly inaccurate solutions to scalar and vector Laplace
problems. Examples can be found in \cite{Arnold2018}, for instance, and also in Section
\ref{sec:numerical-results} of this paper.

The FEEC recipe to preserve the cohomological structure is ensuring that the cohomology
spaces of the discrete and continuous complexes can be isomorphically identified with maps
that are bounded and commute with the operators defined above. By doing so, we can also
ensure that we do not introduce any spurious harmonic scalar/vector fields, as these are
intimately related with the cohomology spaces. In short, we want to construct a discrete
version, \(\complex{R}_h\), of the de Rham complex \(\complex{R}\) (or of its counterpart
with homogeneous boundary conditions, \(\complex{R}_0\)), namely,
\begin{equation*}
	\begin{tikzcd}
		\complex{R}_h~:~ \funcspace{X}^0_h(\domain) \arrow{r}{\grad} & \funcspace{X}^1_h(\domain) \arrow{r}{\curl} & \funcspace{X}^2_h(\domain)\;,
	\end{tikzcd}
\end{equation*}
for some appropriate choice of finite element spaces \(\funcspace{X}^j_h(\domain),\;
j\in\{0,1,2\}\), and where \(h\) is a parameter inversely proportional to the number of
\dofs.

The relevant properties of these finite element spaces have been highlighted in
\cite{Arnold2009,Arnold2018}, and we showcase them here briefly for ease of readability. For
a more comprehensive overview of the properties and corresponding results, we refer the
reader to the works just mentioned. First, we require that the finite element spaces
considered have enough similarity with the continuous ones to be able to approximate them
arbitrarily well. To be concrete, we necessitate that
\[
	\lim_{h \to 0} \inf_{u_h \in \funcspace{X}^j_h}\|u_h -u \| = 0,\quad u \in
	\funcspace{X}^j, j\in \{0,1,2\}.
\]
The second requisite is that the chosen discrete spaces, together with the operators of the
complex \(\complex{R}\), also form a complex, or, in other words,
\[
	\grad \funcspace{X}^0_h \subseteq \funcspace{X}^{1}_h, \quad \curl \funcspace{X}^1_h \subseteq \funcspace{X}^{2}_h,
\]
and the composition of the two operators is zero, as in the continuous case. The final
property we mention here is crucial for the correct behaviour of a discrete solution, and is
the one that establishes a connection between the cohomology spaces of the two complexes.
That is, it should be possible to define a set of bounded projection maps from
\(\complex{R}\) to \(\complex{R}_{h}\) (or from \(\complex{R}_0\) to \(\complex{R}_{h}\)),
in such a way that a commuting diagram is formed. By guaranteeing the three properties
listed above, one can ensure an accurate and stable solution of the mixed formulation of
Laplace problems.

It is precisely this choice of appropriate finite element spaces that will be tackled in the
rest of the paper, more specifically, when the finite element spaces are hierarchical
B-spline spaces. Whether bounded commuting projectors exist for hierarchical B-spline
complexes is still an open issue and outside the scope of the present work. We limit
ourselves to characterizing and fixing the kind of refinements that introduce spurious
harmonic forms, and therefore invalidate the existence of such projectors altogether.
 \section{Spline spaces}\label{sec:spline-spaces}

In this section, all required notation and definitions will be introduced for the univariate
B-splines, and later extended into tensor product B-splines and, finally, the hierarchical
B-spline de Rham complex. We will simplify the notation used in \cite{Evans2018} and
\cite{Shepherd2024} to the two-dimensional case.

\subsection{Univariate B-splines}\label{subsec:univariate-splines}

In order to define univariate B-splines of polynomial degree \(\degree\geq 1\), we need to
define a knot vector. Namely, for \(m \geq 1\), a knot vector \(\knotvec =
(\knot_1,\dots,\knot_{m+\degree+1})\) is a sequence of non-decreasing real numbers, or
knots, satisfying
\[
	\knot_1=\dots=\knot_{\degree}<\knot_{p+1}\leq \dots \leq
	\knot_{m+1}<\knot_{m+2}=\dots=\knot_{m+\degree+1}.
\]
We repeat the first and last knots \(p\) times, which differs from the usual
choice of \(p+1\) repetitions, and enforces homogeneous boundary conditions for splines of
degree \(p\);
we also assume that no knot is repeated more than \(p\) times.
To simplify the notation, we will consider that \(\knot_1=0\) and \(\knot_{m+\degree+1}=1\)
from now on.

We will denote by \(Z\) the set of knots after removing repetitions, in such a way that
\(Z=\{\zeta_1, \dots, \zeta_{z+1}\}\) is the largest subset of \(\knotvec\) such that
\[
	\zeta_i<\zeta_{i+1},\ i\in\{1,\dots,z\}.
\]
This induces a partition of the unit interval, given by the subintervals
\begin{equation}\label{eq:breakpoint-indices}
	\mathcal{I} \coloneqq  \left\{(\zeta_i,\zeta_{i+1}): i\in\{1,\dots,z\}\right\}.
\end{equation}

Using the knot vector \(\knotvec\) we can define two spaces that will be used to establish
discrete complexes.
Let \(\bspspace^{0}(\knotvec)\) be the space of piecewise-polynomial functions of degree
\(\degree\) and regularity \(C^{p-r}\) at knots with multiplicity \(r\) that
vanish at \(\knot \in \{0,1\}\).
Similarly, we set \(\bspspace^{1}(\knotvec)\) as the space of piecewise-polynomial functions
of degree \(\degree-1\) with \(C^{p-r-1}\) regularity at knots with multiplicity \(r\), and
with no conditions imposed at the boundaries.

With the chosen boundary conditions, the dimension of the space \(\bspspace^{j}(\knotvec)\)
is \(m+j, j\in \{0,1\}\), and we can define a set of \(m+j\) basis functions that span it.
As mentioned in the introduction, we will choose to work with the B-splines \(\bsp^j_{i}, i
\in \{1,\dots,m+j\},\) as the basis for \(\bspspace^{j}(\knotvec)\), and denoting the set as
\(\bspbasis^{j}(\knotvec)\). We refer the reader to \cite{Piegl1997, Schumaker2007} for more
details on the construction of B-splines (e.g. using the Cox-de Boor recursion algorithm).
Nonetheless, we highlight a fact about the support of B-splines that will be used
extensively throughout the paper:
\[
	\supp(\bsp^j_{i}) = (\knot_{i}, \knot_{i+p+1-j}),\; j\in \{0, 1\}\;.
\]

\subsection{Tensor Product B-splines}\label{subsec:tensor-product-splines}

Now that we have fully characterized univariate B-spline spaces, we can use them to
construct a two-dimensional tensor-product B-spline space. Note that we will restrict the
definitions and notation to these conditions, but the construction is altogether identical
for higher-dimensional spaces.

Moreover, keeping in line with our desire to work with hierarchical spline spaces, we will
also introduce the relevant notation for the levels of a hierarchy. We will use \(L\) to
refer to the highest non-trivial refinement level and \(0\leq \level \leq L\) for a general
level.

Let \(\bspbasis^{j_1}_{\level}(\knotvec_{\level, 1})\) and
\(\bspbasis^{j_2}_{\level}(\knotvec_{\level, 2})\) be two univariate spline spaces under the
assumptions of \Cref{subsec:univariate-splines}. We denote their associated polynomial
degrees and dimensions as \(p_{(\level, k)}\) and \(m_{(\level, k)}+j_k\), respectively,
where \(j_k \in \{0, 1\}\) and \(k\in\{1,2\}\). We can then define the tensor product space
\[
	\bspspace_{\level}^{\boldvec{j}} \coloneqq \bspspace^{j_1}(\knotvec_{(\level, 1)})
	\otimes \bspspace^{j_2}(\knotvec_{(\level, 2)}).
\]
A basis for this space is given by \[
	\bspbasis^{\boldvec{j}}_\level \coloneqq \left\{
	\bsp^{\boldvec{j}}_{\boldvec{i}, \level}(\xi_1, \xi_2) \coloneqq
	\bsp^{j_1}_{i_1}(\xi_1) \bsp^{j_2}_{i_2}(\xi_2):
	\bsp^{j_k}_{i_k} \in \bspbasis^{j_k}_\level(\knotvec_{(\level, k)})
	\right\}\;.
\]
We have used a bold and upright notation for the two-dimensional multi-indexes
\(\boldvec{j}\) and \(\boldvec{i}\), and will do so throughout the paper. To access their
\(k\)-th entry we respectively write, as above, \(j_k\) and \(i_k\), and we will denote the
\(i_k\)-th knot of the level \(\level\) univariate knot vector \(\knotvec_{(\level, k)}\) as
\(\knot_{i_k, \level, k}\). Lastly, we will also use the notation
\(\abs{\boldvec{j}}\coloneqq \abs{j_1} + \abs{j_2}\) to refer to the 1-norm on vector
spaces.

From this point forward, we will mainly refer to multi-variate basis functions either using
the three indices \(\boldvec{j}, \boldvec{i}, \text{ and } \level\) or omitting all of them
whenever they are not necessary --- unless otherwise specified. We are now capable of defining
a set of finite element spaces such that the discrete complex they give is cohomologically
equivalent to the continuous complex. The relevant spaces are
\[
	\bspspace^0_{\level} \coloneqq \bspspace_{\level}^{(0,0)}\;, \qquad \bspspace^1_{\level}
	\coloneqq \bspspace_{\level}^{(1,0)} \times \bspspace_{\level}^{(0,1)}\;, \qquad
	\bspspace^2_{\level} \coloneqq \bspspace_{\level}^{(1,1)}\;,
\]
where the multi-indices \(\boldvec{j}\) have been explicitly shown for clarity.
The discrete complex follows naturally \cite{Buffa2010, Buffa2011}  \[
	\begin{tikzcd}
		\complex{\bsp}_h~:~ \bspspace^0_{\level} \arrow{r}{\grad} & \bspspace^1_{\level}
		\arrow{r}{\curl} & \bspspace^2_{\level}\;.
	\end{tikzcd}
\]
To help ease notation, we will use \(\boldvec{0}=(0,0)\) throughout the paper. Taking into
account the way in which the \(\grad\) and \(\curl\) operators change the polynomial degree
and regularity of functions it becomes clear that
\[
	\grad(\bspspace^0_{\level}) \subseteq \bspspace^1_{\level}\;,\qquad
	\curl(\bspspace^1_{\level}) \subseteq \bspspace^2_{\level}\;.
\]

Finally, we also define the tensor-product mesh \(\mesh_{\level}\) induced by the two sets
of intervals \(\mathcal{I}_{(\level, 1)}\) and \(\mathcal{I}_{(\level, 2)}\), corresponding
to  \(\knotvec_{(\level, 1)}\) and \(\knotvec_{(\level, 2)}\), respectively. We define the
mesh of level $\level$ as
\[
	\mesh_{\level} = \left\{I_{1}\times I_{2}: I_{1} \in \mathcal{I}_{(\level, 1)}, I_{2}
	\in \mathcal{I}_{(\level, 2)}\right\},
\]
and we will refer to \(\meshel \in \mesh_{\level}\) as an element.

\subsection{Hierarchical B-splines}\label{subsec:hierarchical-splines}

With all the single-level tensor-product spline spaces from
\Cref{subsec:tensor-product-splines} in hand, we will construct a hierarchical B-spline
space. Our focus is primarily introducing the notation required for the main results of this
work. For more details on hierarchical B-spline spaces, their basis functions, and
hierarchical meshes, we point the interested reader to \cite{Bracco2019, Buffa2022}.

We start by defining a set of nested domains and function spaces defined on them,
assigning each domain and corresponding space to a specific refinement level. For the
spaces, we require that \(\bspspace^{\boldvec{j}}_0\subseteq
\bspspace^{\boldvec{j}}_1\subseteq\dots\subseteq \bspspace^{\boldvec{j}}_L\) for all valid
\(\boldvec{j}\). Additionally, the nested set of closed domains \(\domain_\level\subset
\mathbb{R}^2\) need to satisfy
\[
	\domain =: \domain_0\supseteq \domain_1 \supseteq\dots \supseteq \domain_{L}
	\supseteq \domain_{L+1}\coloneqq \emptyset \;.
\]
Furthermore, we also introduce the following assumption on how the subdomains that determine
the mesh can be chosen.
\begin{assum}\label{assum:refinement-domains}
	We assume that \(\domain_{\level+1}\) is given as the union of
	supports of B-splines in \(\bspbasis_\level^{\boldvec{0}}\), that is \[
		\domain_{\level+1} \coloneqq \bigcup_{\bsp \in
		\tilde{\bspbasis}^{\boldvec{0}}_\level}\closedsupp{(\bsp)},\ \level = \{0,\dots,L-1\},
	\]
	where \(\tilde{\bspbasis}^{\boldvec{0}}_\level\subseteq \bspbasis^{\boldvec{0}}_\level\)
	is some set of level \(\level\) B-splines.
\end{assum}

This leads us to defining a hierarchical mesh, \(\mesh\), as
\begin{equation}\label{eq:hier-mesh}
	\mesh \coloneqq \bigcup_{\level \in \{0,\dots,L\}} \left\{ \meshel \in \mesh_\level:
	\meshel \subseteq \domain_\level \wedge \meshel \not \subseteq \domain_{\level+1}
	\right\}.
\end{equation}
As a consequence of \Cref{assum:refinement-domains}, all elements of \(\mesh\) are pairwise
disjoint and \(\bigcup_{\meshel \in \mesh} \overline{\meshel} = \domain\).
Any element \(\meshel \in \mesh\) is denoted as active.

Consequently, a standard way of selecting a basis \(\hbbasis_L^{\boldvec{j}}\) for
hierarchical B-spline spaces uses the following algorithm \cite{Giannelli2012, Buffa2022}:
\begin{enumerate}
	\item \(\hbbasis_0^{\boldvec{j}} \coloneqq \bspbasis^{\boldvec{j}}_0\).
	\item For \(\level = 0,\dots,L-1\), let \(
	      \hbbasis_{\level+1}^{\boldvec{j}} \coloneqq
	      \hbbasis^{\boldvec{j}}_{\level+1, A} \cup \hbbasis^{\boldvec{j}}_{\level+1, B}
	      \), where
	      \begin{align*}
\hbbasis^{\boldvec{j}}_{\level+1, A} & \coloneqq \left\{
		      \bsp \in \hbbasis_{\level}^{\boldvec{j}}: \supp(\bsp) \not \subseteq
		      \domain_{\level+1}
		      \right\},                                                \\
\hbbasis^{\boldvec{j}}_{\level+1, B} & \coloneqq \left\{
		      \bsp \in \bspbasis^{\boldvec{j}}_{\level+1} : \supp(\bsp) \subseteq
		      \domain_{\level+1}
		      \right\}\;.
	      \end{align*}
\end{enumerate}
We then set \(\hbspace^{\boldvec{j}}_{\level} \coloneqq
\bspan{\hbbasis^{\boldvec{j}}_{\level}}\), for any level \(0\leq \level \leq L\).
As we did for elements, we denote a basis function \(\bsp\) as active whenever \(\bsp \in
\hbbasis^{\boldvec{j}}_{L}\), and as inactive whenever the previous condition is not
true.

The corresponding hierarchical spaces are, then,
\[
	\hbspace^0_{\level} \coloneqq \hbspace_{\level}^{(0,0)}\;, \qquad \hbspace^1_{\level}
	\coloneqq \hbspace_{\level}^{(1,0)} \times \hbspace_{\level}^{(0,1)}\;, \qquad
	\hbspace^2_{\level} \coloneqq \hbspace_{\level}^{(1,1)}\;,
\]
and, for each \(\level=0,\dots, L\), the hierarchical complex is
\begin{equation}\label{eq:hierarchical-complex}
	\begin{tikzcd}
		\complex{\hb}_h~:~ \hbspace^0_{\level} \arrow{r}{\grad} & \hbspace^1_{\level}
		\arrow{r}{\curl} & \hbspace^2_{\level}.
	\end{tikzcd}
\end{equation}

\begin{rem} \label{rem:mesh}
	Under \Cref{assum:refinement-domains}, the three spaces in
	\eqref{eq:hierarchical-complex} are associated to the same hierarchical mesh. Moreover,
	under the same assumption, any active element \(\meshel \in \mesh \cap \mesh_\level\) of
	level $\level$, is contained in the support of  at least one active function of the same
	level.
\end{rem}

The basis just defined above for hierarchical spaces can be slightly altered to preserve
partition of unity and to reduce the support of the basis functions through a procedure
called truncation. For this, we introduce the higher-level representation of B-splines.
Namely, given \(\bsp^{\boldvec{j}}_{\boldvec{i}, \level} \in
\bspbasis^{\boldvec{j}}_{\level}\), it is possible to write the basis function as
\begin{equation}\label{eq:finer-rep}
	\bsp^{\boldvec{j}}_{\boldvec{i}, \level} \equiv \sum_{\bsp \in
		\bspbasis^{\boldvec{j}}_{\level+1}} c^{\level,
			\level+1}_{\bsp}(\bsp^{\boldvec{j}}_{\boldvec{i}, \level}) \bsp,\quad c^{\level,
			\level+1}_{\bsp}(\bsp^{\boldvec{j}}_{\boldvec{i}, \level}) \in \mathbb{R}
\end{equation}
due to nestedness of the function spaces. Using this property, we can then define the
truncation operator at level $\ell+1$ as in \cite{Giannelli2012}
\begin{equation}\label{eq:truncated-rep}
	\trunc^{\level+1}(\bsp^{\boldvec{j}}_{\boldvec{i}, \level}) \coloneqq
	\sum_{\substack{\bsp \in \bspbasis^{\boldvec{j}}_{\level+1}:
			\\ \supp(\bsp)\not \subseteq \domain_{\level+1}}}
	c^{\level, \level+1}_{\bsp}(\bsp^{\boldvec{j}}_{\boldvec{i}, \level})\bsp.
\end{equation}
The truncation operation described above is linear and, in particular, can be written as a
matrix multiplication, see \cite{D'Angella2020} for more details.

Using this, the new truncated hierarchical basis is given by
\begin{enumerate}
	\item \(\thbbasis_0^{\boldvec{j}} \coloneqq \bspbasis^{\boldvec{j}}_0\).
	\item For \(\level = 0,\dots,L-1\), let \(
	      \thbbasis_{\level+1}^{\boldvec{j}} \coloneqq
	      \thbbasis^{\boldvec{j}}_{\level+1, A} \cup \hbbasis^{\boldvec{j}}_{\level+1, B}
	      \), where
	      \[
		      \thbbasis^{\boldvec{j}}_{\level+1, A} \coloneqq \left\{
		      \trunc^{\level+1}(\bsp): \bsp\in \thbbasis_{\level}^{\boldvec{j}} ~\wedge~ \supp(\bsp) \not \subseteq \domain_{\level+1}
		      \right\}\;.
	      \]
\end{enumerate}
Both the hierarchical basis and the truncated hierarchical basis span the same space
\cite{Giannelli2012}, or in other words, \(\bspan{\thbbasis^{\boldvec{j}}_{\level}} =
\bspan{\hbbasis^{\boldvec{j}}_{\level}} = \hbspace^{\boldvec{j}}_{\level}\) for all levels
\(\level = 0,\dots, L\).
However, the truncated basis is preferable in some instances. 
By removing redundant terms of the linear combination in \eqref{eq:finer-rep}, in the sense
that some basis functions are already present in the next level, we can reduce the support
of the resulting truncated basis functions, and therefore improve sparsity patterns and
matrix conditioning; for instance of the mass matrices associated to scalar and vector
Laplace problems which are associated to the de Rham complex \cite{Buffa2022}.
Also, the truncated basis recovers the convex partition of unity property of tensor-product
B-splines.

Given a B-spline \(B \in \bspbasis^{\boldvec{j}}_{\level+1}\), from the relation
\eqref{eq:finer-rep} we say that \(\bsp^{\boldvec{j}}_{\boldvec{i}, \level} \in
\bspbasis^{\boldvec{j}}_{\level}\) is a parent function of \(B\) if
\(c^{\level,\level+1}_{\bsp}(\bsp^{\boldvec{j}}_{\boldvec{i}, \level}) \not = 0\). We will
also define the notion of mother of a truncated B-spline \cite{Giannelli2013}. Given a
truncated basis function \(\thb^{\boldvec{j}}_{\boldvec{i}, \level} \in
\thbbasis^{\boldvec{j}}_{\level}\), we say that \(\bsp^{\boldvec{j}}_{\boldvec{i}, \level}\)
is the mother of \(\thb^{\boldvec{j}}_{\boldvec{i}, \level}\) if
\begin{equation} \label{eq:mother-spline}
	\thb^{\boldvec{j}}_{\boldvec{i}, \level} =
	\trunc^{L}\left(\cdots\left(\trunc^{\level+1}\left(\bsp^{\boldvec{j}}_{\boldvec{i},
				\level}\right)\right)\right)\;,
\end{equation}
and we denote this relationship as \(\mot(\thb^{\boldvec{j}}_{\boldvec{i}, \level})
\coloneqq \bsp^{\boldvec{j}}_{\boldvec{i}, \level}\). We will also make use of the trivial
identification \(\mot(\bsp^{\boldvec{j}}_{\boldvec{i}, \level}) =
\bsp^{\boldvec{j}}_{\boldvec{i}, \level}\) for any basis function in
\(\hbbasis^{\boldvec{j}}_{\level}\). Moreover, the definition given by
\eqref{eq:mother-spline} can easily be extended to inactive basis functions, since
\eqref{eq:truncated-rep} is still well-defined at every level.

We mentioned before how spurious harmonic forms can be introduced for arbitrary refinement
domains. The reason for this can be understood by carefully considering the
construction of the hierarchical spaces. The way in which the selection mechanism of
active and inactive basis functions in each level works can, in principle, lead to a discrete
complex with different cohomology spaces than the original one. This can happen because the
coarse basis functions we remove from one level and the finer basis functions we replace
them with in the next level can span spaces with different topological properties. See
\cite[Fig. 7]{Shepherd2024} and \cite[Fig. 13]{Evans2018}, and the surrounding discussions,
for some illuminating examples of this phenomenon.
 \section{Admissibility}\label{sec:admissibility}

In adaptive methods for hierarchical splines, optimal order of convergence is proved by
using the property of admissibility, which limits the difference of level between
interacting functions \cite{Buffa2015,Buffa2017}. We will prove that, under
\Cref{assum:refinement-domains}, admissibility for the first space implies it for all the
spaces of the complex.

\begin{defn}\label{def:admissibility}
	We say that \(\hbspace^{j}_L\), \(j \in \{0,1,2\}\), is \(\hbbasis\)- or
	\(\thbbasis\)-admissible of class \(m \geq 2\) if on any element \(\meshel \in \mesh\),
	the non-vanishing basis functions in \(\hbbasis^{j}_L\) or \(\thbbasis^{j}_L\),
	respectively, belong to at most \(m\) successive levels.
	Concretely, for \(\meshel \in \mesh\) and any \(\bsp^{j}_{i, \level}, \bsp^{j}_{i',
	\level'} \in \hbbasis^{j}_L\), or \(\thbbasis^{j}_L\), we have
	\[
		\left(
			\supp\left(\bsp^{j}_{i, \level}\right) \cap \meshel \neq \emptyset \wedge
			\supp\left(\bsp^{j}_{i', \level'}\right) \cap \meshel \neq \emptyset 
		\right)
		\implies \abs{\level - \level'} < m.
	\]

	By extension, we say that the hierarchical complex \eqref{eq:hierarchical-complex} is
	\(\hbbasis\)- or \(\thbbasis\)-admissible of class \(m\) if the corresponding previous
	condition holds for every \(j \in \{0,1,2\}\).

	If there are no restrictions on the number of interacting levels we say \(m=\infty\).
\end{defn}

The following definition relates a basis function with the ones involved when computing the
differential operators, that we call co-face B-splines. From now on, we will use
\(\boldvec{\delta}_k\) to denote the two-dimensional index with value \(1\) at position
\(k\), and \(0\) in the other position.

\begin{defn}\label{def:int-der-relation} Let \(\bsp^{\boldvec{j}}_{\boldvec{i}, \level} \in
	\bspbasis^{\boldvec{j}}_{\level}\) and \(\bsp^{\boldvec{j}'}_{\boldvec{i}', \level} \in
	\bspbasis^{\boldvec{j}'}_{\level}\). We say that
	\(\bsp^{\boldvec{j}}_{\boldvec{i},\level} \) is a facet B-spline of
	\(\bsp^{\boldvec{j}'}_{\boldvec{i}', \level} \) and, conversely,
	\(\bsp^{\boldvec{j}'}_{\boldvec{i}', \level} \) is a co-face B-spline of
	\(\bsp^{\boldvec{j}}_{\boldvec{i}, \level} \) if, for some \(k \in \{1,2\} \),
	\(\boldvec{j}' - \boldvec{j} = \boldvec{\delta}_k \) and \(	\boldvec{i}' - \boldvec{i} =
	c\boldvec{\delta}_k,\) with \(c \in \{0, 1\}\).
\end{defn}

We also introduce two auxiliary results that relate their support before and after
truncation.
\begin{lem}\label{lem:hb-supp}
	Let \(\bsp^{\boldvec{j}}_{\boldvec{i}, \level} \in
	\bspbasis^{\boldvec{j}}_{\level}\), and
	let \(\bsp^{\boldvec{j}'}_{\boldvec{i}', \level} \in
	\bspbasis^{\boldvec{j}'}_{\level}\) be a co-face B-spline of
	\(\bsp^{\boldvec{j}}_{\boldvec{i}, \level}\).
	Then, it holds that \(
	\supp(\bsp^{\boldvec{j}}_{\boldvec{i}, \level}) \supseteq
	\supp(\bsp^{\boldvec{j}'}_{\boldvec{i}', \level}).
	\)
\end{lem}
\begin{proof}
	The result holds because the co-face B-spline \(\bsp^{\boldvec{j}'}_{\boldvec{i}',
		\level}\) has equal or lower degree, and its local knot vectors are contained in the
	local knot vectors of \(\bsp^{\boldvec{j}}_{\boldvec{i}, \level}\).
\end{proof}
\begin{lem}\label{lem:thb-supp}
	Let \(\thb^{\boldvec{j}}_{\boldvec{i}, \level} \in
	\thbbasis^{\boldvec{j}}_{\level}\) and \(\bsp^{\boldvec{j}}_{\boldvec{i},
		\level}=\mot(\thb^{\boldvec{j}}_{\boldvec{i}, \level})\).
	Also, let \(\bsp^{\boldvec{j}'}_{\boldvec{i}', \level} \in
	\bspbasis^{\boldvec{j}'}_{\level}\) be a co-face B-spline of
	\(\bsp^{\boldvec{j}}_{\boldvec{i}, \level}\), and \(\thb^{\boldvec{j}'}_{\boldvec{i}',
		\level} =
	\trunc^{L}\left(\cdots\left(\trunc^{\level+1}\left(\bsp^{\boldvec{j}'}_{\boldvec{i}',
				\level}\right)\right)\right)\).
	Then, it holds that \[
		\supp(\thb^{\boldvec{j}}_{\boldvec{i}, \level}) \supseteq
		\supp(\thb^{\boldvec{j}'}_{\boldvec{i}', \level}).
	\]
\end{lem}
\begin{proof}
	Let us consider an active element \(\meshel \in \mesh \cap \mesh_{\tilde{\level}}\),
	with \(\tilde{\level}>\level\), where \(\thb^{\boldvec{j}}_{\boldvec{i}, \level}\)
	is truncated, in the sense that \(Q \not \subset \supp
	(\thb^{\boldvec{j}}_{\boldvec{i}, \level})\) and \(Q \subset \supp
	(\bsp^{\boldvec{j}}_{\boldvec{i}, \level})\).
	Let \(\mathcal{R}^{\boldvec{j}}_{Q, \tilde{\level}} \subset
	\bspbasis^{\boldvec{j}}_{\tilde{\level}}\) denote the smallest subset of B-splines
	of level \(\tilde{\level}\) that are linearly combined to represent
	\(\bsp^{\boldvec{j}}_{\boldvec{i}, \level}\) on \(\meshel\).
	Since \(\thb^{\boldvec{j}}_{\boldvec{i}, \level}\) vanishes on \(\meshel\), the
	support of B-splines in \(\mathcal{R}^{\boldvec{j}}_{Q, \tilde{\level}}\) must be
	contained in \(\Omega_{\tilde{\level}}\).

	Next, let \(\mathcal{R}^{\boldvec{j}'}_{\meshel, \tilde{\level}} \subset
	\bspbasis^{\boldvec{j}'}_{\tilde{\level}}\) be the smallest subset of B-splines that
	can be linearly combined to represent \(\bsp^{\boldvec{j}'}_{\boldvec{i}', \level}\)
	on \(\meshel\), and suppose that \(\thb^{\boldvec{j}'}_{\boldvec{i}', \level}\) does
	not vanish on \(\meshel\).
	Meaning, there exists some \(\xbsp' \in
	\mathcal{R}^{\boldvec{j}'}_{\meshel,\tilde{\level}}\) such that \(\supp(\xbsp')
	\not\subset \domain_{\tilde{\level}}\).
	Since \(\bsp^{\boldvec{j}'}_{\boldvec{i}', \level}\) is a co-face B-spline
	of \(\bsp^{\boldvec{j}}_{\boldvec{i}, \level}\), there is some \(\xbsp \in
	\mathcal{R}^{\boldvec{j}}_{\meshel, \tilde{\level}}\) that is a face B-spline of
	\(\xbsp'\).
	However, invoking \Cref{lem:hb-supp}, this would mean  that
	\(\thb^{\boldvec{j}}_{\boldvec{i}, \level}\) is not truncated on \(\meshel\),
	which is a contradiction.
\end{proof}
Finally, we prove that admissibility for \(\hbspace^{0}_{L}\) implies the same property
throughout the complex.

\begin{prop}\label{lem:admissibility}
	Under \Cref{assum:refinement-domains}, if \(\hbspace^{0}_{L}\) is
	\(\hbbasis\)- or \(\thbbasis\)-admissible of  class \(m\) so is \(\hbspace^{j}_{L}\) for
	every \(j \in \{1,2\}\).
\end{prop}
\begin{proof}
	Since we will prove the result for both \(\hbbasis\)- and \(\thbbasis\)-admissibility,
	we will use \(\basis{X}\) to denote the basis in either of the two cases, and \(X\) for
	the corresponding basis functions.

	Let us suppose that some \(\hbspace^{j'}_L\), with \(j' \in \{1,2\}\), is not
	\(\basis{X}\)-admissible of class \(m\). In particular, this means that for some active
	elements \(\meshel \in \mesh\cap \mesh_{\level}\) and \(\tilde{\meshel} \in \mesh\cap
	\mesh_{\tilde{\level}}\), with \(\level, \tilde{\level} \in \{0,1,\dots,L\}\) and
	\(\tilde{\level} > \level+m-1\), we have at least one active level \(\level\) basis
	function \(X^{\boldvec{j}'}_{\boldvec{i}', \level} \in \basis{X}_{L}^{\boldvec{j}'}\),
	with \(\abs{\boldvec{j}'}=j'\), such that
	\begin{equation}\label{eq:thb-element-overlap}
		\supp(X^{\boldvec{j}'}_{\boldvec{i}',\level}) \cap \meshel \neq \emptyset
		\wedge \supp(X^{\boldvec{j}'}_{\boldvec{i}',\level})\cap \tilde{\meshel} \neq
		\emptyset.
	\end{equation}
	We will show that this will lead to a contradiction.	We start by observing that we
	can use \Cref{lem:hb-supp} or \Cref{lem:thb-supp}, to find basis functions such that the
	following chained containment holds,
	\begin{equation}\label{eq:thb-admissibility-containment}
		\supp(X^{\boldvec{0}}_{\boldvec{i}_{\level}, \level}) \supseteq \cdots \supseteq
		\supp(X^{\boldvec{j}}_{\boldvec{i}, \level}) \supseteq
		\supp(X^{\boldvec{j}'}_{\boldvec{i}',\level})\;,
	\end{equation}
	and where if \(X\) is a spline in \Cref{eq:thb-admissibility-containment} and
	\(\hat{X}\) is the spline to the right of \(X\), then \(\mot(X)\) is a face B-spline of
	\(\mot(\hat{X})\).

	It is true that \(X^{\boldvec{0}}_{\boldvec{i}_{\level},\level}\) can't be active, since
	this would violate the assumption on the \(\basis{X}\)-admissibility of
	\(\hbspace^{0}_{L}\) due to \eqref{eq:thb-element-overlap} and
	\Cref{assum:refinement-domains}, see \Cref{rem:mesh}. Therefore, because the support of
	\(X^{\boldvec{0}}_{\boldvec{i}_{\level},\level}\) contains \(\meshel \in \mesh\cap
	\mesh_{\level}\) we know that \(\supp(X^{\boldvec{0}}_{\boldvec{i}_{\level},\level})
	\not\subseteq \domain_{\level}\).

	Let \(X^{\boldvec{0}}_{\boldvec{i}_{\level-1},\level-1}\) be such that
	\(\mot(X^{\boldvec{0}}_{\boldvec{i}_{\level},\level})\) is used in the finer level
	representation of \(\mot(X^{\boldvec{0}}_{\boldvec{i}_{\level-1},\level-1})\), as in
	\eqref{eq:finer-rep}. Then, \(\supp(X^{\boldvec{0}}_{\boldvec{i}_{\level-1},\level-1})
	\cap \meshel \neq \emptyset\) because \(X^{\boldvec{0}}_{\boldvec{i}_{\level},\level}\)
	does not vanish  on \(\meshel\) and, as mentioned above,
	\(\supp(X^{\boldvec{0}}_{\boldvec{i}_{\level},\level}) \not\subseteq \domain_{\level}\).

	Again, if \(X^{\boldvec{0}}_{\boldvec{i}_{\level-1},\level-1}\) were active it would
	lead to a contradiction on the admissibility of \(\hbspace^{0}_{L}\) because of
	\eqref{eq:thb-element-overlap} and \Cref{assum:refinement-domains}. So, we know that
	\(X^{\boldvec{0}}_{\boldvec{i}_{\level-1},\level-1}\) is inactive and
	\(\supp(X^{\boldvec{0}}_{\boldvec{i}_{\level-1},\level-1}) \not\subseteq
	\domain_{\level-1}\). We can repeat these steps recursively until we arrive at some
	\(X^{\boldvec{0}}_{\boldvec{i}_0, 0}\) that is inactive because
	\(\supp(X^{\boldvec{0}}_{\boldvec{i}_{0}, 0}) \not \subseteq \domain_{0} = \domain\),
	which is a contradiction.
\end{proof}
 \section{Exact Meshes}\label{sec:exact-meshes}

We now move on to define precisely when a hierarchical refinement leads to an inexact
complex, and how it can be altered by refining some additional functions to form an exact
complex that satisfies the sufficient conditions in \cite{Shepherd2024}. To begin, we
introduce some notation to help the readability of the forthcoming proofs. First, we define
the subset \(\Bll \subseteq \bspbasis^{\boldvec{0}}_{\level}\) composed of B-splines in
level \(\level\) whose support is contained in \(\domain_{\level+1}\), i.e.,
\begin{equation*}\label{eq:bll-definition}
	\Bll \coloneqq \left\{\bsp \in \bspbasis^{\boldvec{0}}_{\level}:
	\supp(\bsp) \subseteq \domain_{\level+1}\right\}\;.
\end{equation*}
Second, we will simplify the notation of B-spline basis functions, and write
\(\xbsp_{\boldvec{i}} \coloneqq \bsp^{\boldvec{0}}_{\boldvec{i}, \level} \), since all
arguments in the proofs rely only on B-splines in the first space of the de Rham complex and
two consecutive hierarchical levels.

Now, we recall the definitions of chains and shortest chains from \cite{Shepherd2024}, and
we introduce some particular cases of shortest chains.

\begin{defn}[Chain]\label{def:chain}
	Let \(\xbsp_{\boldvec{i}}, \xbsp_{\boldvec{j}} \in \Bll\).
	We say that there is a chain between \(\xbsp_{\boldvec{i}}\) and \(\xbsp_{\boldvec{j}}\)
	if there exists \(r \in \mathbb{Z}_{\geq 0}\) and a sequence of B-splines \(\xbsp_{\boldvec{t}_l} \in
	\Bll,~l = \{0,\dots,r\}\), such that \(\xbsp_{\boldvec{i}} = \xbsp_{\boldvec{t}_0},
	\xbsp_{\boldvec{j}} = \xbsp_{\boldvec{t}_r}\) and \(|\boldvec{t}_{l} - \boldvec{t}_{l-1}| = 1\)
	for all \(l \in \{1,\dots,r\}\).
	We will denote this chain by \(\chain{\boldvec{i}}{\boldvec{j}} :=
	\{\xbsp_{\boldvec{t}_l} \in \Bll~:~l = 0,\dots,r\}\), and we will call \(r\) the
	length of the chain.
\end{defn}
\begin{defn}[Shortest chain]\label{def:shortest-chain}
	Let \(\xbsp_{\boldvec{i}}, \xbsp_{\boldvec{j}} \in \Bll\). A chain between
	\(\xbsp_{\boldvec{i}}\) and \(\xbsp_{\boldvec{j}}\), as in \Cref{def:chain}, is said to
	be a shortest chain if we have \(r= \sum_k \abs{j_k - i_k}\).
\end{defn}

\begin{defn}[Direction-\(k\) chain]\label{def:k-chain}
	Let \(\xbsp_{ \boldvec{i} }, \xbsp_{\boldvec{j}} \in \Bll\) and \(k \in \{1,2\}\) be
	such that \(i_{k'} = j_{k'}\) for \(k' \not = k\).
	We say that there is a direction-\(k\) chain between \(\xbsp_{\boldvec{i}}\) and
	\(\xbsp_{\boldvec{j}}\) if \(\xbsp_{\boldvec{r}} \in \Bll\) for every \(\boldvec{r} \in
	\mathbb{N}^2\) such that \[
		\min\{i_k,j_k\} < r_k < \max\{i_k,j_k\}
		~\wedge~
		i_{k'} = r_{k'} = j_{k'} \text{ for } k'  \neq k.
	\]
\end{defn}

\begin{defn}[\lchain]\label{def:L-chain}
	Let \(\xbsp_{\boldvec{i}}, \xbsp_{\boldvec{j}} \in \Bll \).
	We say that there exists an \lchain between them if for some \(\xbsp_{\boldvec{c}} \in
	\Bll\) there exists a direction-\(k_1\) chain, \(\chain{\boldvec{i}}{\boldvec{c}}\),
	between \(\xbsp_{\boldvec{i}}\) and \(\xbsp_{\boldvec{c}}\) and a direction-\(k_2\)
	chain, \(\chain{\boldvec{c}}{\boldvec{j}}\), between \(\xbsp_{\boldvec{c}}\) and
	\(\xbsp_{\boldvec{j}}\), such that \(k_1 \neq k_2\).
	The union of the chains \(\chain{\boldvec{i}}{\boldvec{c}}\) and
	\(\chain{\boldvec{c}}{\boldvec{j}}\) is called the \lchain between
	\(\xbsp_{\boldvec{i}}\) and \(\xbsp_{\boldvec{j}}\), and \(\xbsp_{\boldvec{c}}\) is
	called its corner element.
\end{defn}

The previous definitions are crucial to the main contribution of this paper,
\Cref{alg:exact-mesh}, and in particular \lchains will be our main tool for turning inexact
complexes exact through additional refinement. As it was proved in
\cite{Shepherd2024}, the lack of exactness is related to problematic pairs (see
\Cref{def:problematic}): refined functions that are close to each other, but not connected
through a shortest chain of refined functions. Connecting problematic pairs through an
L-chain fixes the pair, and limits the amount of new potentially problematic pairs.

The next results show that the previous two types of chains are shortest chains.

\begin{lem}\label{lem:k-chain-shortest}
	Let \(\xbsp_{\boldvec{i}}, \xbsp_{\boldvec{j}} \in \Bll \) have a direction-\(k\) chain
	\(\chain{\boldvec{i}}{\boldvec{j}}\) between them.
	Then, \(\chain{\boldvec{i}}{\boldvec{j}}\) is a shortest chain.
\end{lem}
\begin{proof}
	Let \(\boldsymbol{\Delta} = \boldvec{j} -\boldvec{i}\) and assume \wlg that
	\(\Delta_k\geq0\). Then, the direction-\(k\) chain \(\chain{\boldvec{i}}{\boldvec{j}}\) is
	equal to
	\(
	\{
	\xbsp_{\boldvec{i}},
	\xbsp_{\boldvec{i}+\boldvec{\delta}_k},\xbsp_{\boldvec{i}+2\boldvec{\delta}_k},\dots,
	\xbsp_{\boldvec{j}}
	\}.
	\)
	It is clear that this chain satisfies the conditions for being a shortest chain, since
	its length is \(\abs{\boldsymbol{\Delta}} = j_k - i_k\).
\end{proof}

\begin{lem}\label{lem:L-chain-shortest}
	Let \(\xbsp_{\boldvec{i}}, \xbsp_{\boldvec{j}} \in \Bll \). Then, any \lchain between
	them is a shortest chain.
\end{lem}
\begin{proof}
	Assume without loss of generality that \(j_k \geq i_k\), for \(k = 1,2\). Since the
	\lchain is the union of a direction-\(1\) and direction-\(2\) chain, it is easy to see
	that its length is \((j_1 - i_1) + (j_2 - i_2) = \abs{\boldvec{j} - \boldvec{i}}\), and
	hence it is a shortest chain.
\end{proof}

The following are two definitions and the main result from \cite{Shepherd2024}, that gives a
sufficient condition for exactness.
\begin{defn}\label{def:nlintersec}
	Let \( \xbsp_{\boldvec{i}}, \xbsp_{\boldvec{j}} \in \Bll \).
	We say that \(\xbsp_{\boldvec{i}}\) and \(\xbsp_{\boldvec{j}}\) share a \nlintersec if
	there exists some \(k_0 \in\{1, 2\}\) and
	\(\{\knot_{t_1,\level+1,1},\knot_{t_2,\level+1,2}\} \in
	\knotvec_{\level+1,1} \times \knotvec_{\level+1,2}\) such that
	\begin{align*}
		\closedsupp{(\xbsp_{\boldvec{i}})} \cap \closedsupp{(\xbsp_{\boldvec{j}})} & \supseteq
		\bigtimes_{k=1}^2 I_k\;,                                                               \\
		I_k                                                                        & \coloneqq
		\begin{cases*}
			(\knot_{t_k,\level+1,k}, \knot_{t_k+p_{(\level+1, k)},\level+1,k})\;,
			                                         & \( k\neq k_0\)\;, \\
			\left\{\knot_{t_k,\level+1,k}\right\}\;, & \( k = k_0\)\;.
		\end{cases*}
	\end{align*}
\end{defn}

\begin{defn}[Problematic pairs and chains] \label{def:problematic}
	Let \( \xbsp_{\boldvec{i}}, \xbsp_{\boldvec{j}} \in \Bll \) share a \nlintersec. We say
	that the pair is problematic if there is no shortest chain between them.
	We also say that a chain is problematic with \( \xbsp_{\boldvec{i}}
	\) if any spline in the chain is problematic with \( \xbsp_{\boldvec{i}} \).
\end{defn}

In \Cref{fig:problematic-pair-illustration} we illustrate the definition of problematic
pairs by showing pairs of basis functions in three different scenarios.

\begin{figure}[ht]
	\centering
	\hfill
	\begin{subfigure}[t]{0.32\textwidth}
		\centering

 		\caption{}
		\label{fig:problematic-pair-algorithm-case-1-illustration}
	\end{subfigure}
	\begin{subfigure}[t]{0.32\textwidth}
		\centering
		%
 		\caption{}
		\label{fig:problematic-pair-algorithm-case-2-illustration}
	\end{subfigure}
	\begin{subfigure}[t]{0.32\textwidth}
		\centering
		%
 		\caption{}
		\label{fig:problematic-pair-algorithm-case-3-illustration}
	\end{subfigure}
	\hfill
	\strut
	\caption{
		Illustration of three different scenarios, all using \(p_{(\level, k)} = 2\) for all
		\(\level\) and \(k\). In the left figure there is no problematic pair, in the centre
		figure there is a problematic pair and in the last figure there is a \nlintersec and
		a shortest chain, so the pair is not problematic. The supports of
		\(\xbsp_{\boldvec{i}}\) and \(\xbsp_{\boldvec{j}}\) are coloured in grey, with their
		intersection in a darker shade. Also, the possible \(I_k\) contained in the
		intersection of the supports of \(\xbsp_{\boldvec{i}}\) and \(\xbsp_{\boldvec{j}}\),
		as in \Cref{def:nlintersec}, are highlighted in bold, black lines. Finally, filled
		dots represent the indices \(\boldvec{i}\) and \(\boldvec{j}\) and hollow dots the
		indices of the other B-splines in \(\Bll\).
	}
	\label{fig:problematic-pair-illustration}
\end{figure}

\begin{thm}[\cite{Shepherd2024}, Theorem 4.23]\label{thm:exact-complex}
	Let us assume that, for every \( \xbsp_{\boldvec{i}}, \xbsp_{\boldvec{j}} \in \Bll \)
	sharing a \nlintersec, there exists a shortest chain between them, for every \(\level =
	0, \ldots, L\). Then the hierarchical complex \eqref{eq:hierarchical-complex} is exact
	for \(\level = 0, \ldots, L\).
\end{thm}

\begin{rem}
	We note here that \cite[Assumption 3]{Shepherd2024} is more restrictive due to the
	higher-dimensional setting. In the conditions of this paper, pair-wise checks are enough
	to ensure exactness of the de Rham complex.
\end{rem}

To fix problematic pairs it is necessary to join them by a shortest chain of refined
B-splines. However, additional refinement can lead to new problematic pairs that need to be
checked and resolved recursively. The following results show that \lchains reduce the number
of checks to be performed during this process.
\begin{defn}\label{def:resolved}
	Let \(\xbsp_{\boldvec{i}} \in \Bll\) and \(k \in \{1,2\}\). We define the \(k\)-side configuration of
	\(\xbsp_{\boldvec{i}}\) as \[
		\sideconfig_{\boldvec{i}, k} \coloneqq \left\{
		\xbsp_{\boldvec{t}} \in \bspbasis^{\boldvec{0}}_{\level} :
		\, \boldvec{t} - \boldvec{i} = \pm \boldvec{\delta}_k,
		\right\}\;.
	\]
	Moreover, we say that \(\xbsp_{\boldvec{i}} \in \Bll\) is resolved in direction \(k\) if
	\( \sideconfig_{\boldvec{i}, k} \subseteq \Bll\).
\end{defn}

\begin{lem}[Problematic sides]\label{lem:problematic-sides}
	Let \(\xbsp_{\boldvec{i}},\xbsp_{\boldvec{j}} \in \Bll \) be a problematic pair, and
	\(\xbsp_{\boldvec{i}}\) be resolved in direction \(k \in \{1,2\}\). Then, there must be
	some \(\xbsp_{\boldvec{t}} \in \sideconfig_{\boldvec{i}, k}\) that is problematic with
	\(\xbsp_{\boldvec{j}}\), and \(\abs{t_k - j_k} < \abs{i_k - j_k}\).
\end{lem}
\begin{proof}
	Since \(\xbsp_{\boldvec{i}}\) is resolved in direction \(k\), by definition there exists
	\(\xbsp_{\boldvec{t}}\in \sideconfig_{\boldvec{i},k}\) such that \(\abs{t_k - j_k} <
	\abs{i_k - j_k}\). This implies that \(\closedsupp{(\xbsp_{\boldvec{i}})} \cap
	\closedsupp{(\xbsp_{\boldvec{j}})} \subseteq \closedsupp{(\xbsp_{\boldvec{t}})} \cap
	\closedsupp{(\xbsp_{\boldvec{j}})}\), which in turn means that there is a \nlintersec
	between \(\xbsp_{\boldvec{t}}\) and \(\xbsp_{\boldvec{j}}\). Since a shortest chain
	between \(\xbsp_{\boldvec{t}}\) and \(\xbsp_{\boldvec{j}}\) would result in a shortest
	chain between \(\xbsp_{\boldvec{i}}\) and \(\xbsp_{\boldvec{j}}\), and hence a
	contradiction, the pair \(\xbsp_{\boldvec{t}}, \xbsp_{\boldvec{j}}\) must be
	problematic.
\end{proof}

\begin{lem}[Problematic chain in 2D] \label{lem:problematic-chain}
	Let \(\xbsp_{\boldvec{i}},\xbsp_{\boldvec{j}}, \xbsp_{\boldvec{l}}  \in \Bll\) be such
	that there exists an \lchain \( \chain{\boldvec{i}}{\boldvec{j}} \) between \(
	\xbsp_{\boldvec{i}} \) and \(\xbsp_{\boldvec{j}}\), with corner element \(
	\xbsp_{\boldvec{c}}\). If \(\xbsp_{\boldvec{l}}\) is problematic with
	\(\chain{\boldvec{i}}{\boldvec{j}}\), then it must be problematic with at least one of
	\(\xbsp_{\boldvec{i}}, \xbsp_{\boldvec{j}} \) and \(\xbsp_{\boldvec{c}}\).
\end{lem}
\begin{proof}
	Suppose \( \xbsp_{\boldvec{l}} \) is problematic with some \(
	\xbsp_{\boldvec{t}} \in \chain{\boldvec{i}}{\boldvec{j}}\). If \( \xbsp_{\boldvec{t}}
	\in \{\xbsp_{\boldvec{i}}, \xbsp_{\boldvec{j}}, \xbsp_{\boldvec{c}}\} \) the proof is
	trivial, so, consider \( \xbsp_{\boldvec{t}} \in \chain{\boldvec{i}}{\boldvec{j}}
	\backslash \{\xbsp_{\boldvec{i}}, \xbsp_{\boldvec{j}}, \xbsp_{\boldvec{c}}\}\).
	Then, since \(\xbsp_{\boldvec{t}}\) will be resolved in some direction \(k\), we can
	recursively apply \Cref{lem:problematic-sides} to prove that there exists
	\(\xbsp_{\boldvec{v}} \in \{\xbsp_{\boldvec{i}}, \xbsp_{\boldvec{j}},
	\xbsp_{\boldvec{c}}\} \) that is problematic with \(\xbsp_{\boldvec{l}}\).
\end{proof}

We now introduce the concept of an interaction box, and some related results that will help
to simplify the search for shortest chains in the algorithms of \Cref{sec:algorithms}.

\begin{defn}\label{def:int-box}
	Let \(\xbsp_{\boldvec{i}} \in \Bll\). We define the interaction box
	\(\interbox{\boldvec{i}}\) of \(\xbsp_{\boldvec{i}}\) as the set of
	B-splines in \(\Bll\) that are at a maximum distance of \(\degree_{(\ell, k)}+1\) knot
	indices, in each dimension. In other words, \[
		\interbox{\boldvec{i}} \coloneqq \left\{
		\xbsp_{\boldvec{j}} \in \Bll : \abs{i_k - j_k} \leq \degree_{(\ell, k)} + 1, \text{ for } k = 1,2
		\right\}.
	\]
\end{defn}

To refer to the intersection of two interaction boxes \(\interbox{\boldvec{i}}\)
and \(\interbox{\boldvec{j}}\) we will write \(\interbox{\boldvec{i}}{\boldvec{j}}\coloneqq
\interbox{\boldvec{i}}\cap\interbox{\boldvec{j}}\).

\begin{lem}\label{lem:int-box-chain}
	Let \(\xbsp_{\boldvec{i}},\xbsp_{\boldvec{j}} \in \Bll\). Any shortest chain between
	\(\xbsp_{\boldvec{i}}\) and \(\xbsp_{\boldvec{j}}\) is contained in
	\(\interbox{\boldvec{i}}{\boldvec{j}}\). Moreover, if
	\(\chain{\boldvec{i}}{\boldvec{j}}\) is a chain contained in
	\(\interbox{\boldvec{i}}{\boldvec{j}}\), then there exists a shortest chain between
	\(\xbsp_{\boldvec{i}}\) and \(\xbsp_{\boldvec{j}}\).
\end{lem}
\begin{proof}
	The first statement is trivial and follows from the definition of the interaction box.
	If a subsequence of \(\chain{\boldvec{i}}{\boldvec{j}}\) is a shortest chain then we are
	done. Suppose then that no subset of \(\chain{\boldvec{i}}{\boldvec{j}}\) is a shortest
	chain and assume, \wlg, that \(i_k<j_k\), for \(k =1,2\). By the chain not being a
	shortest chain, we know that there must exist at least one pair,
	\(\xbsp_{\boldvec{l}},\xbsp_{\boldvec{t}} \in \chain{\boldvec{i}}{\boldvec{j}} \cap
	\interbox{\boldvec{i}}{\boldvec{j}}\), with a ``kink''. By this, we mean that there
	exists \(k \in \{1,2\}\) such that \(l_{k} = t_{k}\) but the corresponding
	direction-\(k'\) chain \(\chain{\boldvec{l}}{\boldvec{t}}\), \(k' \in \{1, 2\}
	\backslash k\), is not a  subsequence of \(\chain{\boldvec{i}}{\boldvec{j}}\) — see
	\Cref{fig:kink-example} for an example.
	\begin{figure}[ht]
		\centering
		\hfill
		\begin{subfigure}[t]{0.35\textwidth}
			\centering
			\begin{tikzpicture}[scale=0.4]
\coordinate (B_i) at (1.0, 1.0);
	\coordinate (B_j) at (5.0, 5.0);
	\fill[gray, opacity=0.6] (B_i) rectangle (5.0, 5.0);
	\fill[gray, opacity=0.6] (B_j) rectangle (9.0, 9.0);

\foreach \x in {0.0,...,10.0} 
		\draw[black, line width=1.0pt] (\x, 0.0) -- (\x, 10.0);
	\foreach \y in {0.0,...,10.0} 
		\draw[black, line width=1.0pt] (0.0, \y) -- (10.0, \y);

\foreach \x in {0.0,...,7.0} 
		\draw[black, line width=1.0pt] (\x + 1.5, 1.0) -- (\x + 1.5, 9.0);
	\foreach \y in {0.0,...,7.0} 
		\draw[black, line width=1.0pt] (1.0, \y + 1.5) -- (9.0, \y + 1.5);

\fill[black] (B_i) circle (0.2);
	\node[anchor=north west] at (B_i) {\(\boldvec{i}\)};
	\fill[black] (1.0, 2.0) circle (0.2);
	\fill[black] (2.0, 2.0) circle (0.2);
	\fill[black] (3.0, 2.0) circle (0.2);
	\draw[black, line width=1.0pt] (3.0, 2.0) circle (0.3);
	\draw[black, line width=1.0pt] (4.0, 2.0) circle (0.3);
	\draw[black, line width=1.0pt] (5.0, 2.0) circle (0.3);
	\fill[black] (3.0, 1.0) circle (0.2);
	\fill[black] (4.0, 1.0) circle (0.2);
	\fill[black] (5.0, 1.0) circle (0.2);
	\fill[black] (5.0, 2.0) circle (0.2);
	\fill[black] (5.0, 3.0) circle (0.2);
	\fill[black] (5.0, 4.0) circle (0.2);
	\fill[black] (B_j) circle (0.2);
	\node[xshift=0.1cm, yshift=0.1cm] at (B_j) {\(\boldvec{j}\)};
\end{tikzpicture}
 			\caption{\(\chain{\boldvec{i}}{\boldvec{j}}\)}
			\label{fig:kink-example-a}
		\end{subfigure}
		\begin{subfigure}[t]{0.35\textwidth}
			\centering
			\begin{tikzpicture}[scale=0.4]
\coordinate (B_i) at (1.0, 1.0);
	\coordinate (B_j) at (5.0, 5.0);
	\fill[gray, opacity=0.6] (B_i) rectangle (5.0, 5.0);
	\fill[gray, opacity=0.6] (B_j) rectangle (9.0, 9.0);

\foreach \x in {0.0,...,10.0} 
		\draw[black, line width=1.0pt] (\x, 0.0) -- (\x, 10.0);
	\foreach \y in {0.0,...,10.0} 
		\draw[black, line width=1.0pt] (0.0, \y) -- (10.0, \y);

\foreach \x in {0.0,...,7.0} 
		\draw[black, line width=1.0pt] (\x + 1.5, 1.0) -- (\x + 1.5, 9.0);
	\foreach \y in {0.0,...,7.0} 
		\draw[black, line width=1.0pt] (1.0, \y + 1.5) -- (9.0, \y + 1.5);

\fill[black] (B_i) circle (0.2);
	\node[anchor=north west] at (B_i) {\(\boldvec{i}\)};
	\fill[black] (1.0, 2.0) circle (0.2);
	\fill[black] (2.0, 2.0) circle (0.2);
	\fill[black] (3.0, 2.0) circle (0.2);
	\fill[black] (4.0, 2.0) circle (0.2);
	\fill[black] (5.0, 2.0) circle (0.2);
	\fill[black] (5.0, 3.0) circle (0.2);
	\fill[black] (5.0, 4.0) circle (0.2);
	\fill[black] (B_j) circle (0.2);
	\node[xshift=0.1cm, yshift=0.1cm] at (B_j) {\(\boldvec{j}\)};
\end{tikzpicture}
 			\caption{\(\widetilde{C}_{\boldvec{i},\boldvec{j}}\)}
			\label{fig:kink-example-b}
		\end{subfigure}
		\hfill
		\strut
		\caption{
			Example of a chain, \(\chain{\boldvec{i}}{\boldvec{j}}\), with a “kink”, and its
			counterpart, \(\widetilde{C}_{\boldvec{i},\boldvec{j}}\), with the “kink”
			removed. Shaded regions represent the supports of the bicubic B-splines
			\(\xbsp_{\boldvec{i}}\) and \(\xbsp_{\boldvec{j}}\), filled dots the indices of
			functions in each chain, and hollow dots the indices of functions in
			\(\interbox{\boldvec{i}}{\boldvec{j}}\) that can be used to remove the “kink”.
		}
		\label{fig:kink-example}
	\end{figure}

	However, since \(\xbsp_{\boldvec{l}}\) and \(\xbsp_{\boldvec{t}}\) are both in
	\(\interbox{\boldvec{i}}{\boldvec{j}}\), their supports have a non-empty intersection,
	which implies that \(\supp(\xbsp_{\boldvec{r}}) \subseteq
	\closedsupp{(\xbsp_{\boldvec{l}})} \cup \closedsupp{(\xbsp_{\boldvec{t}})} \subseteq
	\Omega_{\level+1} \) for all \(\boldvec{r} \in \NN^{2}\) such that
	\[
		l_{k} = r_{k} = t_{k} ~\wedge~
		\min(l_{k'},t_{k'}) < r_{k'} < \max(l_{k'},t_{k'})\;.
	\]
	Hence, all the B-splines in \(\chain{\boldvec{l}}{\boldvec{t}}\) belong to \(\Bll\), and
	as a consequence to \(\interbox{\boldvec{i}}{\boldvec{j}}\), and we can use them to
	remove the ``kink'' between \(\xbsp_{\boldvec{l}}\) and \(\xbsp_{\boldvec{t}}\) to
	obtain a new chain \(\widetilde{C}_{\boldvec{i},\boldvec{j}}\) with a smaller length.
	Repeating the procedure for every such ``kink'' will yield a shortest chain between
	\(\xbsp_{\boldvec{i}}\) and \(\xbsp_{\boldvec{j}}\).
\end{proof}

\begin{lem}\label{lem:containment}
	Let \(\xbsp_{\boldvec{i}},\xbsp_{\boldvec{j}} \in \Bll\) share a \nlintersec and have an
	\lchain in \(\chain{\boldvec{i}}{\boldvec{j}}\) between them. Moreover, let
	\(\xbsp_{\boldvec{l}}\) share a \nlintersec with some \(\xbsp_{\boldvec{t}} \in
	\chain{\boldvec{i}}{\boldvec{j}}\) and such that there is an \lchain
	\(\chain{\boldvec{l}}{\boldvec{t}}\). Then, there are no problematic pairs between
	\(\chain{\boldvec{i}}{\boldvec{j}}\) and \(\chain{\boldvec{l}}{\boldvec{t}}\).
\end{lem}
\begin{proof}
	We will prove the result directly, by showing that there is a shortest chain between any
	pair of B-splines, one in each \lchain, that share a \nlintersec.

	As such, suppose that \(\xbsp_{\boldvec{p}} \in \chain{\boldvec{i}}{\boldvec{j}}\) and
	\(\xbsp_{\boldvec{p}'}\in \chain{\boldvec{l}}{\boldvec{t}}\) share a \nlintersec. Since
	\(\xbsp_{\boldvec{i}}\) and \(\xbsp_{\boldvec{j}}\) share a \nlintersec, the same is
	true for any other pair of B-splines in \(\chain{\boldvec{i}}{\boldvec{j}}\), and in
	particular for \(\xbsp_{\boldvec{p}}\) and \(\xbsp_{\boldvec{t}}\). Analogously,
	\(\xbsp_{\boldvec{p}'}\) and \(\xbsp_{\boldvec{t}}\) must also share a \nlintersec,
	since both belong to \(\chain{\boldvec{l}}{\boldvec{t}}\). Hence, we can conclude that
	\(\xbsp_{\boldvec{t}} \in \interbox{\boldvec{p}}{\boldvec{p}'} \).

	Using the definition of \lchains, it is clear that we can find subsequences of
	\(\chain{\boldvec{i}}{\boldvec{j}}\) and \(\chain{\boldvec{l}}{\boldvec{t}}\), contained
	in \(\interbox{\boldvec{p}}{\boldvec{p}'}\), that are chains between
	\(\xbsp_{\boldvec{t}}\) and \(\xbsp_{\boldvec{p}}\) and between  \(\xbsp_{\boldvec{t}}\)
	and \(\xbsp_{\boldvec{p}'}\), respectively. Invoking \Cref{lem:int-box-chain}, we reason
	that there is a shortest chain between \(\xbsp_{\boldvec{p}}\) and
	\(\xbsp_{\boldvec{p}'}\), thus finishing the proof.
\end{proof}

\begin{rem}\label{rem:preferable-lchain}
	Let \(\xbsp_{\boldvec{i}},\xbsp_{\boldvec{j}}\in \Bll\) be a problematic pair. There are
	two possible \lchains that can be constructed between the pair --- depending on the corner
	element chosen. It is always preferable to select the \lchain that leads to the highest
	number of resolved B-splines, as this reduces the necessary checks for new problematic
	pairs.
\end{rem}

In this section, we have shown that \lchains can be used to create exact complexes, by using
them to connect problematic pairs, thus creating a shortest chain between them as required
by the conditions of \Cref{thm:exact-complex}. In the next section, we will detail how the
theoretical results of this section can be used to implement an algorithm to identify and
solve problematic pairs through refinement procedures.
 \section{Algorithms}\label{sec:algorithms}

We will now describe a set of algorithms that can be easily incorporated into standard
numerical solvers, and that ensure that the resulting discrete de Rham complex remains exact
after local refinement.

The main algorithm is \Cref{alg:exact-mesh}, which takes as input a hierarchical space with
no problematic pairs, described by its basis \(\thbbasis^{\boldvec{0}}_{L}\), and a set of
elements marked for refinement, with the implicit assumption that both
\(\thbbasis^{\boldvec{0}}_{L}\) and the marked elements respect
\Cref{assum:refinement-domains}. The algorithm returns an updated basis
\(\thbbasis^{\boldvec{0}}_{L_*}\) corresponding to an exact complex, where we note that
\(L_*\) is either \(L\) or \(L+1\), depending on the marked elements. From an implementation
standpoint, we consider that the objects \(\thbbasis^{\boldvec{0}}_{L}\) and
\(\thbbasis^{\boldvec{0}}_{L_*}\) have information about the full single-level
tensor-product basis at each level, as well as the corresponding hierarchical mesh.

At every level, this algorithm performs the following steps. We first refine the marked
elements given as input, and from those we create the set of possibly problematic pairs that
need to be checked, this will be detailed in \cref{alg:initiate-pairs}. We then enter an
iterative \var{while} loop to remove the problematic pairs. For each problematic pair we
create a shortest chain in the form of an \lchain, selected according to
\cref{rem:preferable-lchain} using the method \var{get\_lchain\_corner}, and then refine the
corresponding corner functions for all the selected \lchains. Since the refinement of corner
functions may cause the appearance of new problematic pairs, we generate a new set of pairs
to be checked at the next step of the \var{while} loop, exploiting
\cref{lem:problematic-chain} to consider only those pairs which contain at least one of the
refined corner functions. At the end of the loop there are no problematic pairs at the
current level, but \Cref{assum:refinement-domains} may not be satisfied if the support of
the refined corner functions is not contained in \(\domain_{\level - 1}\). To repair this,
we loop over the newly refined corners to check their support, and if the condition is
violated we use the method \var{get\_a\_parent\_func} to choose a parent function that will
add the least amount of refinement, and modify the marked elements of the previous level to
ensure that its support will be refined. Since the refinement of the parent function may
again violate \Cref{assum:refinement-domains} in previous levels, we also store it to check
its support at the next step. Finally, after looping through all levels we update the
hierarchical basis using the new refined mesh.

\begin{algorithm}
	\caption{Exact mesh refinement.}
	\label{alg:exact-mesh}
	\begin{algorithmic}[1]
		\Require \(\thbbasis^{\boldvec{0}}_L\), \var{marked\_els}.
		\Ensure Refined \(\thbbasis^{\boldvec{0}}_{L_*}\) with no problematic pairs.

		\State \(\var{previous\_parents} \gets \emptyset \)
		\ForAll{\(\level \var{ in } \{L,\dots,0\}\)}
		\State \(\thbbasis^{\boldvec{0}}_L \gets \var{refine\_mesh}(\thbbasis^{\boldvec{0}}_L, \level, \var{marked\_els}[\level])\)
		\State \(\var{unchecked\_pairs} \gets
		\var{initiate\_pairs}(\thbbasis^{\boldvec{0}}_L,\level,\var{marked\_els}[\level])\)
		\State \(\var{problematic\_mesh} \gets \var{true}\)
		\State \(\var{level\_corners} \gets \emptyset\)
		\While{\var{problematic\_mesh}}
		\State  \(\var{problematic\_mesh} \gets \var{false}\)
		\State \(\var{current\_corners} \gets \emptyset\)
		\ForAll{\var{pair in unchecked\_pairs}}
		\If{\(\neg\var{is\_problematic}(\thbbasis^{\boldvec{0}}_L, \level, \var{pair})\)}
		\State \var{continue}
		\EndIf
		\Statex
		\State \var{corner} \(\gets\)
		\var{get\_lchain\_corner}(\(\thbbasis^{\boldvec{0}}_L, \level, \var{pair}\))
		\State \var{current\_corners} \(\gets\) \var{current\_corners} \(\cup\)
		\var{corner}
		\State  \(\var{problematic\_mesh} \gets \var{true}\)
		\EndFor
		\Statex
		\State \(\thbbasis^{\boldvec{0}}_L \gets \var{refine\_mesh}(\thbbasis^{\boldvec{0}}_L, \level, \var{get\_support(current\_corners)})\)
		\State \var{unchecked\_pairs} \(\gets\)
		\var{get\_local\_pairs}(\(\thbbasis^{\boldvec{0}}_L, \level,\) \var{current\_corners})
		\State \(\var{level\_corners} \gets \var{level\_corners} \cup \var{current\_corners}\)
		\EndWhile
		\Statex
		\State \(\var{level\_corners} \gets \var{level\_corners} \cup \var{previous\_parents}\)
		\State \(\var{refined\_elements} \gets \var{marked\_els}[\level-1] \cup
		\var{get\_subdomain}(\thbbasis^{\boldvec{0}}_L, \level - 1)\)
		\State
		\(\var{previous\_parents} \gets \emptyset \)
		\ForAll{\var{func in level\_corners}}
		\If{\(\neg\var{is\_supported\_on}(\thbbasis^{\boldvec{0}}_L, \level-1,
			\var{refined\_elements}, \var{func})\)}
		\State \(\var{parent\_func} \gets
		\var{get\_a\_parent\_func(} \var{func)}\)
		\State \(\var{previous\_parents} \gets \var{previous\_parents} \cup \var{parent\_func} \)
		\State \var{marked\_els}\([\level-1] \gets \var{marked\_els} [\level-1] \cup \var{get\_support(parent\_func)}\)
		\EndIf
		\EndFor
		\EndFor
		\Statex
		\State \(\thbbasis^{\boldvec{0}}_{L_*}\gets\)
		\var{update\_space(\(\thbbasis^{\boldvec{0}}_{L}\))}
		\Statex
		\State \Return \(\thbbasis^{\boldvec{0}}_{L_*}\)
	\end{algorithmic}
\end{algorithm}

To generate the possibly problematic pairs,  \cref{alg:initiate-pairs} loops over the basis
functions in \(\Bll\). According to \Cref{lem:problematic-sides} we can limit ourselves to
pairs formed by unresolved basis functions, that are selected using the method
\var{is\_resolved}. Moreover, since the input mesh is assumed to have no problematic pairs,
every problematic pair must originate from what has been added to \(\Bll\) after refinement,
i.e., the supports intersect the marked elements. Once we have collected all such functions,
we call \var{get\_local\_pairs} to ensure the locality of the generated pairs. This method,
detailed in \Cref{alg:get_local_pairs}, loops over all the previous unresolved functions
(the third input argument) and, for each, stores pairs of other unresolved functions in its
interaction box. Finally, we use \var{unique} to remove repeated pairs.

To check if a given pair is problematic, \cref{alg:is-problematic-pair} calls
\cref{alg:nlintersec} to check for a \nlintersec, and \cref{alg:shortest-chain} to check if
there is a shortest chain between them.

\cref{alg:nlintersec} performs the check for a \nlintersec between the pair using
\cref{def:nlintersec}. It first calls the method \var{get\_support\_per\_dim}, that computes
the endpoints of the intervals where a basis function is supported, in each dimension, as
defined in \eqref{eq:breakpoint-indices}. Concretely, if the support of \(\xbsp\) in
dimension \(k\) contains only the intervals
\((\zeta_i,\zeta_{i+1}),\dots,(\zeta_j,\zeta_{j+1})\), for some \(i,j\), then
\var{get\_support\_per\_dim} would return the integers \(\{i,\dots,j+1\}\). Then, the
algorithm computes the intersection of the supports of the pair of basis functions. By
working with the integer indices of the knots, the subsequent computations are cheaper than
if using real values. Another method used in this algorithm that should be implemented is
\var{get\_contained\_indices}, which returns the biggest subset of \(\knotvec_{(\level+1,
	k)}\) contained in the intersection of the pair's supports, for a given dimension \(k\) and
level \(\level\).

Then, \Cref{alg:shortest-chain} is used to check if a pair of basis functions has a shortest
chain connecting them. We assume that the given pair of basis functions shares a
\nlintersec, since this method is only called after \cref{alg:nlintersec} returns a true
condition. The check for a direction-\(k\) chain between the pair is trivial. If no such
chain exists, we then construct the interaction box, as in \Cref{def:int-box}, and check if
there is a shortest chain between the pair using \Cref{lem:int-box-chain}. To do this, we
chose to implement the interaction box as a graph --- in this case, a subgraph of a lattice
graph --- and then use a standard algorithm to check if a path exists between nodes on a
graph, named as \var{has\_chain} in the algorithm. Packages that perform these operations
are commonplace and should be available in most programming languages.

\begin{rem}
	\Cref{alg:exact-mesh} can be adapted to enforce an admissible
	\(\hbspace^{\boldvec{0}}_{L}\) with few alterations. For a given refinement level
	\(\level\), the imposition of admissibility only affects elements at levels
	\(\tilde{\level}\), with \(\tilde{\level} < \level\). Therefore, we can impose
	admissibility after fixing all the problematic pairs at a level \(\level\). This way, we
	ensure that any possible problematic pair introduced by the addition of admissibility
	will still be fixed in the following iterations.
\end{rem}

\begin{rem}
	Algorithm \ref{alg:exact-mesh} could be further optimized by using
	\Cref{lem:containment} to reduce the number of pairs to be checked. We do not do this
	here to keep the description of the algorithm simple.
\end{rem}

\begin{algorithm}
	\caption{\var{initiate\_pairs}}
	\label{alg:initiate-pairs}
	\begin{algorithmic}[1]
		\Require
		\(\thbbasis^{\boldvec{0}}_{L}\),
		level \(\level\),
		\(\var{marked\_els}[\level]\)
		\Ensure Pairs of basis functions that might be problematic.
		\State \(\var{unchecked} \gets \emptyset\)
		\ForAll{\(\xbsp_{\boldvec{i}} \in \Bll \)} \Comment{As defined in \Cref{sec:exact-meshes}.}
		\If{\(\neg \var{is\_resolved}(\xbsp_{\boldvec{i}})\) and
			\(\closedsupp{(\xbsp_{\boldvec{i}})} \cap \var{marked\_els}[\level] \neq \emptyset\)}
		\State \var{unchecked} \(\gets \var{unchecked} \cup
		\xbsp_{\boldvec{i}}\)
		\Comment{Using \Cref{lem:problematic-sides}.}
		\EndIf
		\EndFor
		\Statex
		\State \( \var{unchecked\_pairs} \gets \var{get\_local\_pairs}(\thbbasis^{\boldvec{0}}_{L}, \level, \var{unchecked}) \)
		\Statex

		\State \Return \var{unchecked\_pairs}
	\end{algorithmic}
\end{algorithm}

\begin{algorithm}
	\caption{\var{get\_local\_pairs}}
	\label{alg:get_local_pairs}
	\begin{algorithmic}[1]
		\Require
		\(\thbbasis^{\boldvec{0}}_{L}, \level\),
		\(\var{unchecked}\)
		\Ensure Pairs of basis functions that might be problematic.

		\State \(\var{pairs} \gets \emptyset \)
		\ForAll{\(\xbsp_{\boldvec{i}} \in \var{unchecked}\)}
		\ForAll{\(\xbsp_{\boldvec{j}} \in \interbox{\boldvec{i}}\)}
		\If {\(\neg \var{is\_resolved}(\xbsp_{\boldvec{j}})\)}
		\State \(\var{pairs} \gets \var{pairs} \cup (\xbsp_{\boldvec{i}}, \xbsp_{\boldvec{j}})\)
		\EndIf
		\EndFor
		\EndFor
		\Statex
		\State \(\var{pairs} \gets \var{unique(pairs)}\)
		\Statex
		\State \Return \var{pairs}
	\end{algorithmic}
\end{algorithm}

\begin{algorithm}
	\caption{\var{is\_problematic} (Checks if a pair of B-splines is problematic.)}
	\label{alg:is-problematic-pair}
	\begin{algorithmic}[1]
		\Require
		\(\thbbasis^{\boldvec{0}}_{L}\), level \(\level\),
		\var{pair} = \(
		(\xbsp_{\boldvec{i}}, \xbsp_{\boldvec{j}})
		\).
		\Ensure Boolean indicating if the pair is problematic.
		\If{\(\neg\var{has\_minimal\_intersection}(\thbbasis^0_{L}, \level, \var{pair})\)} \Comment{See \cref{alg:nlintersec}}
		\State \Return \var{false}
		\EndIf
		\Statex
		\State \Return \(\neg\var{has\_shortest\_chain}(\thbbasis^0_{L}, \level, \var{pair})\) \Comment{See \cref{alg:shortest-chain}}
	\end{algorithmic}
\end{algorithm}

\begin{algorithm}
	\caption{\var{has\_minimal\_intersection} (Checks if there is a \nlintersec between a pair of
		B-splines.)}
	\label{alg:nlintersec}
	\begin{algorithmic}[1]
		\Require \(\thbbasis^{\boldvec{0}}_{L}\), level \(\level\), \var{pair}
		= \( (\xbsp_{\boldvec{i}}, \xbsp_{\boldvec{j}}) \).
		\Ensure Boolean indicating if the pair shares a \nlintersec.
		\State \var{supp\_bi} \(\gets\) \var{get\_support\_per\_dim}
		(\(\xbsp_{\boldvec{i}}\))
		\State \var{supp\_bj} \(\gets\) \var{get\_support\_per\_dim}
		(\(\xbsp_{\boldvec{j}}\))
		\ForAll{\(k \var{ in } \{1, 2\}\)}
		\State \var{supp\_intersection} \(\gets\) \var{supp\_bi[k]} \(\cap\)
		\var{supp\_bj[k]}
		\State \(I_k \gets\) \var{get\_contained\_indices}(\var{supp\_intersection}, \(\knotvec_{(\level+1,k)}\))
		\Comment{At level \(\ell+1\).}
		\State \var{length\_flag[k]} \(\gets\)
		\var{length(\(I_k\))\;>\;\(p_{(\level+1,k)}\)}
		\Comment{Condition of \Cref{def:nlintersec}.}
		\EndFor
		\Statex
		\State \Return \var{any(length\_flag)}
	\end{algorithmic}
\end{algorithm}

\begin{algorithm}
	\caption{
		\var{has\_shortest\_chain} (Checks if a pair of B-splines that share a
		\nlintersec have a shortest chain between them.)
	}
	\label{alg:shortest-chain}
	\begin{algorithmic}[1]
		\Require
		\(\thbbasis^{\boldvec{0}}_{L}\),
		level \(\level\),
		\var{pair} = \(
		(\xbsp_{\boldvec{i}}, \xbsp_{\boldvec{j}})
		\).
		\Ensure Boolean indicating if there is a shortest chain between the pair of
		B-splines.
		\If{\var{any}(\(\boldvec{i}-\boldvec{j}\))==0}
		\State \Return \var{true}
		\Comment{There is a direction-\(k\) chain.}
		\EndIf
		\Statex
		\State \Return \var{has\_chain(\(\interbox{\boldvec{i}}{\boldvec{j}}\), \(\boldvec{i}\), \(\boldvec{j}\))}
		\Comment{Using \Cref{def:int-box} and \Cref{lem:int-box-chain}.}

	\end{algorithmic}
\end{algorithm}

\begin{thm}
	Given a hierarchical basis \(\thbbasis^{\boldvec{0}}_L\) satisfying
	\Cref{assum:refinement-domains} and without problematic pairs,
	\cref{alg:exact-mesh} returns a basis which satisfies the same properties.
\end{thm}
\begin{proof}
	Let \(\thbbasis^{\boldvec{0}}_L\) and \var{marked\_els} be the inputs to
	\cref{alg:exact-mesh}.
	Since the outermost \var{for} loop in the algorithm runs from \(L\) to \(0\), it
	obviously terminates.
	Therefore, we need to show the following:
	\begin{enumerate}
		\item The \var{while} loop terminates and fixes all problems at level \(\level\).
		\item The \var{level\_corners} \var{for} loop is enough to ensure nestedness and
		      \cref{assum:refinement-domains}.
	\end{enumerate}

	Let us consider the iteration at some level \(\level \in \{0,\dots,L\}\). The first set
	of possibly problematic pairs is given by \Cref{alg:initiate-pairs}. We start by
	discarding the resolved functions using \cref{lem:problematic-sides}, since fixing all
	the pairs with these functions will solve the problems with the resolved ones as well.
	Also, given our assumption that \(\thbbasis^{\boldvec{0}}_L\) is initially without
	problematic pairs, every possible new problematic pair must contain at least of function
	whose support's closure intersects \(\var{marked\_els}[\level]\) — which contain all the
	input marked elements and parents from the previous step.

	Inside the \var{while} loop, we will check every pair in \var{unchecked\_pairs} using
	the method \var{is\_problematic}. The first part checks if a \nlintersec is present and
	the second checks if a shortest chain exists using \cref{lem:int-box-chain}. If a given
	pair is problematic, we fix it, using \cref{lem:L-chain-shortest}, by adding an \lchain
	between the pair, and storing its corner function, which is enough to refine the \lchain
	due to proximity of functions in \cref{def:nlintersec} and the support of B-splines. We
	also set \var{problematic\_mesh} to \var{true}, since every added \lchain might
	introduce new problems. After fixing all problems in \var{unchecked\_pairs}, we refine
	the mesh with all the stored corners, and create the new set of \var{unchecked\_pairs}.
	Since only the recently added corners can introduce new problematic pairs, we again call
	\var{get\_local\_pairs} on \var{current\_corners}. Thus, we fix every problem at level
	\(\level\). To see that the \var{while} loop finishes, notice that if a given set
	\var{unchecked\_pairs} has no problems, then \var{problematic\_mesh} will stay as
	\var{false} and the loop will finish. It is also guaranteed that eventually no more
	problematic pairs will exist — the worst case scenario being a tensor-product refinement
	of the current level.

	Now, we proceed to show that the final mesh satisfies \cref{assum:refinement-domains}.
	After quitting the \var{while} loop, we have refined the mesh with the support of all
	the corners added while solving the problematic pairs, possibly breaking
	\cref{assum:refinement-domains} for \(\domain_{\level}\). Therefore, we must loop over
	all the functions in \var{level\_corners} and parents from the previous step and check
	if their support is contained in \(\domain_{\level-1}\); if not, we choose a parent
	function, store it, and update \(\var{marked\_els}[\level-1]\) with its support. This
	now ensures that \(\domain_{\level}\) satisfies \cref{assum:refinement-domains}.

	In the end, because the hierarchical mesh has been updated with all the needed \lchains
	at every level, we simply need to call \var{update\_space} to obtain a basis with no
	problematic pairs, according to \Cref{lem:L-chain-shortest,thm:exact-complex}.
\end{proof}

\begin{rem}
	We note here that \Cref{alg:exact-mesh} can be easily included in an adaptive refinement
	scheme with minimal effort.
	By adding our proposed algorithm between the \emph{refine} and \emph{solve} steps of an
	adaptive loop \cite{Buffa2022}, we can ensure that the discrete de Rham complex is exact
	at each step.
\end{rem}
 \section{Numerical Results}\label{sec:numerical-results}

In order to show that the theory holds numerically, we have devised some tests for the
non-homogeneous vector Laplace problem and the Maxwell eigenvalue problem. We focus on these
tests since the homogeneous boundary conditions that arise in the weak formulations of these
problems, both essential and natural, imply the non-existence of harmonic scalar fields in
simply connected domains.

For all the ensuing examples, we chose to work with \(h\)-refinement for primarily two
reasons. First, it makes for more visually compelling examples of what sorts of refinements
are problematic or not. Second, using \(p\)-refinement precludes the use of maximally smooth
B-splines, which is one of the major advantages of IGA. Despite this, all the theoretical
results of \Cref{sec:exact-meshes} and the algorithmic results of \Cref{sec:algorithms}
still hold for \(p\)-refinement schemes. Moreover, we will use the same polynomial degree in
every direction. Thus, from now on we will write \(\degree_{(\level, k)} = \degree\) for
every level \(\level\) and direction \(k\).

All numerical tests were implemented using the software package \emph{Mantis.jl}
\cite{C.Cabanas2025} \footnote{The full implementation details of the numerical results that
follow can be found
\href{https://github.com/MantisFEM/Research/tree/paper/2026/L-chain}{here}.}.

\subsection{Vector Laplace problem}

Let the domain  \(\domain \subseteq \mathbb{R}^2\) be the unit square \([0,1]^2\). Since the
domain is simply connected, the continuous de Rham complex has no harmonic vector fields,
and as such, neither should our discretized complex. Considering the mixed weak formulation
of the vector Laplace problem with natural, homogeneous boundary conditions,  the relevant
objects we wish to find are \(\scltrial \in \Hone\) and \(\vectrial \in \Hcurl\) such that
\begin{alignat*}{2}
	\langle \scltrial, \scltest\rangle - \langle \vectrial, \grad\scltest \rangle & = 0,\quad
	                                                                              &                                            & \forall \scltest \in \Hone,                               \\
	\langle \grad\scltrial, \vectest \rangle + \langle \curl \vectrial, \curl \vectest
	\rangle                                                                       & = \langle \forcing, \vectest \rangle,\quad &                             & \forall \vectest\in \Hcurl.
\end{alignat*}
Since the mixed weak-formulation has no boundary conditions imposed, we will use the
notation \(\tilde{\hbspace}^j_{\level}\) to denote the hierarchical space corresponding to
\(\hbspace^j_{\level}\) constructed with open knot-vector univariate B-splines, that is,
with \(p+1\) repetitions of the first and last knots, instead of the conditions used in
\Cref{subsec:univariate-splines}.

The discrete version of the problem is then to find \(\scltrial_h \in
\tilde{\hbspace}^{0}_{L}\) and \(\vectrial_h \in \tilde{\hbspace}^{1}_{L}\) such that
\begin{alignat*}{2}
	\langle \scltrial_h, \scltest_h \rangle - \langle \vectrial_h, \grad\scltest_h \rangle
	        & = 0,\quad                                      &  & \forall \scltest_h \in \tilde{\hbspace}^{0}_{L}\;, \\ \langle
	\grad\scltrial_h, \vectest_h \rangle + \langle \curl \vectrial_h, \curl \vectest_h
	\rangle & = \langle \forcing, \vectest_h \rangle\;,\quad &  & \forall \vectest_h \in
	\tilde{\hbspace}^{1}_{L},
\end{alignat*}
with a choice of polynomial degree \(p=3\) for \textit{Test 1} and \(p=2\) for \textit{Test
2}.

\paragraph{Test 1} For the first test, we will take as the analytical solution
\begin{equation}
	\label{eq:1-form-example-solution}
	\solution(x, y) =
.
\end{equation}
We will solve the problem on two different meshes --- one exact and the other not --- just to
illustrate the spurious results that can arise from  refinement that does not respect the
conditions for exactness. The two meshes are shown in \Cref{fig:1-form-example}. We will
denote the computed solution using the \Cref{fig:1-form-example-a,fig:1-form-example-b}
meshes as \(\vectrial^a_h\) and \(\vectrial^b_h\), respectively.
\begin{figure}[htbp]
	\centering
	\hfill
	\begin{subfigure}[b]{0.35\textwidth}
		\centering
		%
 		\caption{}
		\label{fig:1-form-example-a}
	\end{subfigure}
	\hfill
	\begin{subfigure}[b]{0.35\textwidth}
		\centering
		%
 		\caption{}
		\label{fig:1-form-example-b}
	\end{subfigure}
	\hfill
	\strut

	\caption{
		Two hierarchical meshes for Test 1. Both meshes were created from the same set of
		marked elements, the difference being the use of \cref{alg:exact-mesh} in
		\Cref{fig:1-form-example-b}.
	}
	\label{fig:1-form-example}
\end{figure}

It is easy to see that \(u \in \tilde{\hbspace}^{1}_{L}\), and therefore we should be able
to compute the solution exactly, up to machine precision, in the case of
\Cref{fig:1-form-example-b}. However, due to the problematic intersections in
\Cref{fig:1-form-example-a} we expect to introduce spurious harmonic vector fields. This is
evidenced after computing the corresponding \(L^2\)-norm errors, as we get the values
\begin{equation*}
	\|\vectrial^a_h - \vectrial \| \approx 0.024901,
	\quad
	\|\vectrial^b_h - \vectrial \| \approx 1.303681\times10^{-15}.
\end{equation*}
The appearance of spurious harmonic vector fields introduced by the problematic
intersections can also be seen in \Cref{fig:1-form-example-harmonics}.
\begin{figure}[htbp]
	\centering
	\def\sidelength{3.5cm}
	\begin{minipage}[t]{0.48\textwidth}
		\centering
		\begin{tikzpicture}
			\begin{axis}[
					enlargelimits=false,
					x=\sidelength, y=\sidelength,
					axis on top,
					colorbar, colorbar/width=0.2cm,
					point meta min=1.1e-06, point meta max=0.17,
					colormap/jet,
					colorbar style={
						at={(1.05,0.5)}, anchor=west,
						ytick={0, 0.1, 0.17},
						scaled y ticks=true,
						ticklabel style={font=\small},
					},
				]
				\addplot graphics [
					xmin=0, xmax=1,
					ymin=0, ymax=1,
				] {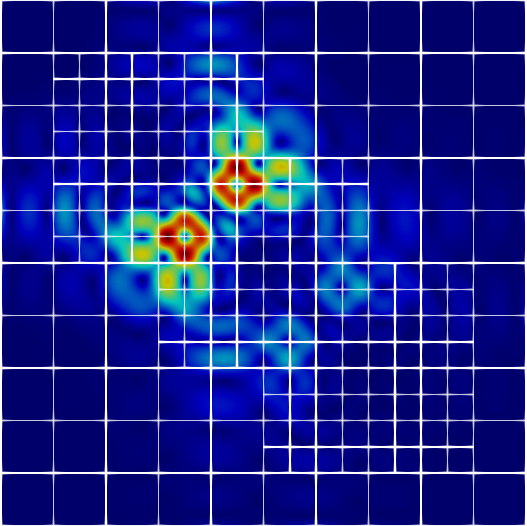};
			\end{axis}
		\end{tikzpicture}
		\caption{
			Plot of \(\|\vectrial^a_h - \vectrial\|\), for \(\vectrial\) as in
			\cref{eq:1-form-example-solution}, showcasing the spurious harmonic vector
			fields.
		}
		\label{fig:1-form-example-harmonics}
	\end{minipage}\hfill
	\begin{minipage}[t]{0.48\textwidth}
		\centering
		\begin{tikzpicture}
			\begin{axis}[
					enlargelimits=false,
					x=\sidelength, y=\sidelength,
					axis on top,
					colorbar, colorbar/width=0.2cm,
					point meta min=0, point meta max=0.012,
					colormap/jet,
					colorbar style={
						at={(1.05,0.5)}, anchor=west,
						ytick={0, 0.006, 0.012},
						scaled y ticks=false,
						yticklabels={0, 0.6, 1.2},
						ylabel={\(\times10^{-2}\)},
						ylabel style={
							rotate=-90,
							at={(-2.0,1.0)},
							anchor=south,
						},
					},
				]
				\addplot graphics [
					xmin=0, xmax=1,
					ymin=0, ymax=1,
				] {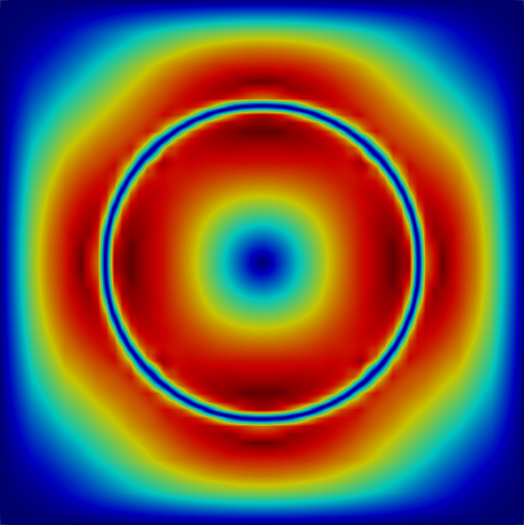};
			\end{axis}
		\end{tikzpicture}
		\caption{
			Magnitude of the analytical solution \cref{eq:1-form-adaptive-solution}.
		}
		\label{fig:1-form-adaptive-solution}
	\end{minipage}
\end{figure}

\paragraph{Test 2} In the second test, the analytical solution is
\begin{equation}
	\label{eq:1-form-adaptive-solution}
	\solution(x,y) =
	\begin{bmatrix}
		2 y^2 (y-1)^2 (x^2-x)(2x-1) \tanh(100((x-0.5)^2+ (y-0.5)^2-0.3^2)) \\
		2 x^2 (x-1)^2 (y^2-y)(2y-1) \tanh(100((x-0.5)^2+ (y-0.5)^2-0.3^2))
	\end{bmatrix}.
\end{equation}
To get a better sense of the circular feature present in the problem, we show a plot of the
analytical solution in \Cref{fig:1-form-adaptive-solution}.

We will perform an adaptive refinement scheme, using the analytical solution to calculate
the \(L^2\)-error per element and a Dörfler parameter \(\theta=0.25\) to mark elements for
refinement.
This parameter determines the cut-off error from which elements are no longer marked.
In other words, we will mark all elements whose local \(L^2\)-error is at least
\((1-\theta)\times100\,\%\) of the maximum element-wise error.
After this marking, we choose as refinement domains the supports of all basis functions
whose supports intersect the Dörfler-marked elements, thus ensuring
\cref{assum:refinement-domains}.
We will start the adaptive loop with a tensor-product mesh and space, and stop the loop
after a pre-determined number of adaptive steps \(N\).
Note that \(N\) may differ from the final number of levels \(L\).
This is enough for our purposes, since we only wish to highlight the appearance of spurious
harmonics in standard refinement schemes that do not consider structure-preservation.
We refer the reader to \cite{Doerfler1996, Buffa2017} for extra details.

\Cref{fig:1-form-adaptive-problematic,fig:1-form-adaptive-lchain}  demonstrate the
difference between the adaptive refinement techniques both with and without the introduction
of \Cref{alg:exact-mesh}.
There, it is easy to see that not only can wrong solutions arise when exact refinements are
not considered during the refinement process, causing a wrong behaviour of the adaptive
algorithm, but also the fact that the number of adaptive steps required to solve problematic
intersections is arbitrary.

\begin{figure}[htbp]
	\centering
	\def\subfigwidth{0.3}
	\def\sidelength{2.8cm}
	\begin{subfigure}[t]{\subfigwidth\textwidth}
		\centering
		\begin{tikzpicture}
			\begin{axis}[
					enlargelimits=false,
					xtick=\empty,
					ytick=\empty,
					x=\sidelength, y=\sidelength,
					axis on top,
					colorbar, colorbar/width=0.1cm,
					point meta min=0, point meta max=6200,
					colormap/jet,
					colorbar style={
						at={(1.05,0.5)}, anchor=west,
						ytick={0, 2000, 4000, 6200},
						yticklabels={0, 2, 4, 6.2},
						ylabel={\(\times10^{3}\)},
						ylabel style={
							rotate=-90,
							at={(-5.0,1.0)},
							anchor=south,
						},
					},
				]
				\addplot graphics [
					xmin=0, xmax=1,
					ymin=0, ymax=1,
				] {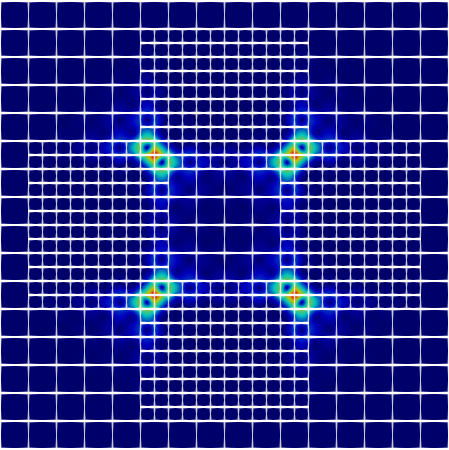};
			\end{axis}
		\end{tikzpicture}
		\caption{Step 1.}
	\end{subfigure}
	\begin{subfigure}[t]{\subfigwidth\textwidth}
		\centering
		\begin{tikzpicture}
			\begin{axis}[
					enlargelimits=false,
					xtick=\empty,
					ytick=\empty,
					x=\sidelength, y=\sidelength,
					axis on top,
					colorbar, colorbar/width=0.1cm,
					point meta min=0, point meta max=6200,
					colormap/jet,
					colorbar style={
						at={(1.05,0.5)}, anchor=west,
						ytick={0, 2000, 4000, 6200},
						yticklabels={0, 2, 4, 6.2},
						ylabel={\(\times10^{3}\)},
						ylabel style={
							rotate=-90,
							at={(-5.0,1.0)},
							anchor=south,
						},
					},
				]
				\addplot graphics [
					xmin=0, xmax=1,
					ymin=0, ymax=1,
				] {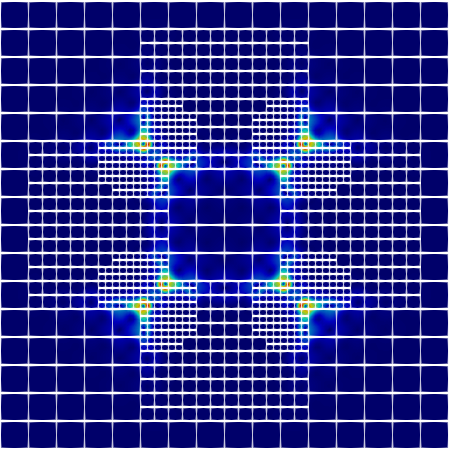};
			\end{axis}
		\end{tikzpicture}
		\caption{Step 2.}
	\end{subfigure}
	\begin{subfigure}[t]{\subfigwidth\textwidth}
		\centering
		\begin{tikzpicture}
			\begin{axis}[
					enlargelimits=false,
					xtick=\empty,
					ytick=\empty,
					x=\sidelength, y=\sidelength,
					axis on top,
					colorbar, colorbar/width=0.1cm,
					point meta min=0, point meta max=2.1e-3,
					colormap/jet,
					colorbar style={
						at={(1.05,0.5)}, anchor=west,
						ytick={0, 0.5e-3, 1e-3, 1.5e-3, 2.1e-3},
						scaled y ticks=false,
						yticklabels={0, 0.5, 1, 1.5, 2.1},
						ylabel={\(\times10^{-3}\)},
						ylabel style={
							rotate=-90,
							at={(-5.0,1.0)},
							anchor=south,
						},
					},
				]
				\addplot graphics [
					xmin=0, xmax=1,
					ymin=0, ymax=1,
				] {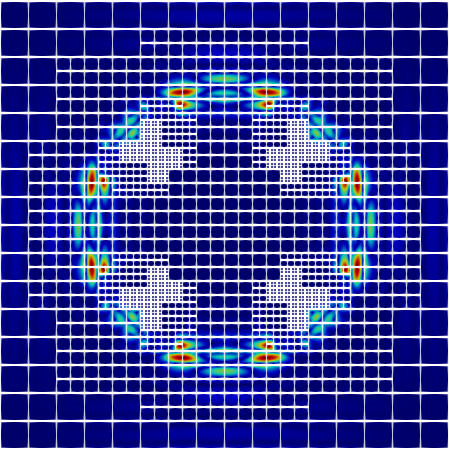};
			\end{axis}
		\end{tikzpicture}
		\caption{Step 3.}
	\end{subfigure}
	\caption{
		Plot of \(\|\vectrial_h - \vectrial\|\) in adaptive refinement scheme without
		\cref{alg:exact-mesh} for the vector Laplace problem with analytical solution
		\eqref{eq:1-form-adaptive-solution}.
		After the first refinement step four problematic intersections are introduced.
		Moreover, despite the fact that the second step correctly refines the regions with
		the highest error, the spurious harmonics have not been resolved, so the overall
		error stays the same.
		By mere chance, the problematic intersections are corrected in Step 3.
		We highlight the different orders of magnitude in the first two sub-figures and the
		last one.
	}
	\label{fig:1-form-adaptive-problematic}
\end{figure}

\begin{figure}[htbp]
	\centering
	\def\subfigwidth{0.3}
	\def\sidelength{2.8cm}
	\begin{subfigure}[t]{\subfigwidth\textwidth}
		\centering
		\begin{tikzpicture}
			\begin{axis}[
					enlargelimits=false,
					xtick=\empty,
					ytick=\empty,
					x=\sidelength, y=\sidelength,
					axis on top,
					colorbar, colorbar/width=0.1cm,
					point meta min=0, point meta max=3.6e-3,
					colormap/jet,
					colorbar style={
						at={(1.05,0.5)}, anchor=west,
						ytick={0, 1e-3, 2e-3, 3e-3, 3.6e-3},
						yticklabels={0, 1, 2, 3, 3.6},
						scaled y ticks=false,
						ylabel={\(\times10^{-3}\)},
						ylabel style={
							rotate=-90,
							at={(-5.0,1.0)},
							anchor=south,
						},
					},
				]
				\addplot graphics [
					xmin=0, xmax=1,
					ymin=0, ymax=1,
				] {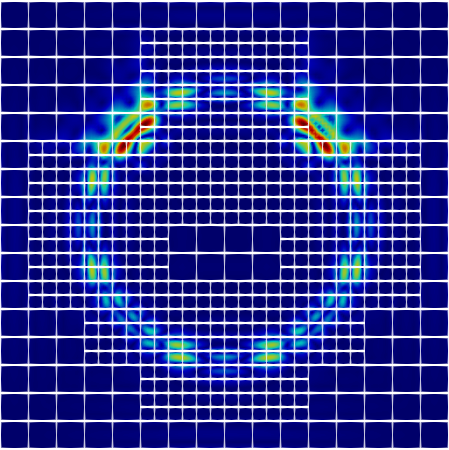};
			\end{axis}
		\end{tikzpicture}
		\caption{Step 1.}
	\end{subfigure}
	\begin{subfigure}[t]{\subfigwidth\textwidth}
		\centering
		\begin{tikzpicture}
			\begin{axis}[
					enlargelimits=false,
					xtick=\empty,
					ytick=\empty,
					x=\sidelength, y=\sidelength,
					axis on top,
					colorbar, colorbar/width=0.1cm,
					point meta min=0, point meta max=2.2e-3,
					colormap/jet,
					colorbar style={
						at={(1.05,0.5)}, anchor=west,
						ytick={0, 0.5e-3, 1e-3, 1.5e-3, 2.2e-3},
						yticklabels={0, 0.5, 1, 1.5, 2.2},
						scaled y ticks=false,
						ylabel={\(\times10^{-3}\)},
						ylabel style={
							rotate=-90,
							at={(-5.0,1.0)},
							anchor=south,
						},
					},
				]
				\addplot graphics [
					xmin=0, xmax=1,
					ymin=0, ymax=1,
				] {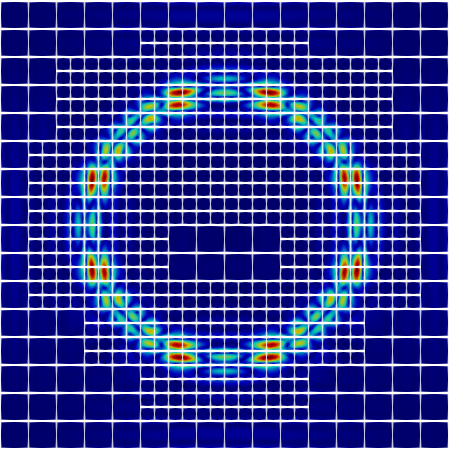};
			\end{axis}
		\end{tikzpicture}
		\caption{Step 2.}
	\end{subfigure}
	\begin{subfigure}[t]{\subfigwidth\textwidth}
		\centering
		\begin{tikzpicture}
			\begin{axis}[
					enlargelimits=false,
					xtick=\empty,
					ytick=\empty,
					x=\sidelength, y=\sidelength,
					axis on top,
					colorbar, colorbar/width=0.1cm,
					point meta min=0, point meta max=1.4e-3,
					colormap/jet,
					colorbar style={
						at={(1.05,0.5)}, anchor=west,
						ytick={0, 0.5e-3, 1e-3, 1.4e-3},
						yticklabels={0, 0.5, 1, 1.4},
						scaled y ticks=false,
						ylabel={\(\times10^{-3}\)},
						ylabel style={
							rotate=-90,
							at={(-5.0,1.0)},
							anchor=south,
						},
					},
				]
				\addplot graphics [
					xmin=0, xmax=1,
					ymin=0, ymax=1,
				] {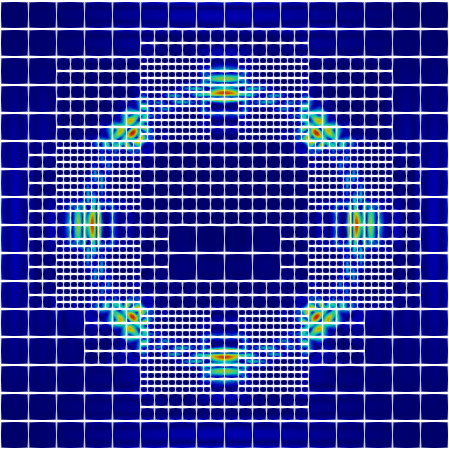};
			\end{axis}
		\end{tikzpicture}
		\caption{Step 3.}
	\end{subfigure}
	\caption{
		Plot of \(\|\vectrial_h - \vectrial\|\) in adaptive refinement scheme with
		\cref{alg:exact-mesh} for the vector Laplace problem with analytical solution
		\eqref{eq:1-form-adaptive-solution}.
		By adding \lchains according to \cref{alg:exact-mesh}, we avoid the problematic
		intersections encountered in \Cref{fig:1-form-adaptive-problematic} and arrive at a
		final mesh with a refinement pattern that is much more natural given the circular
		feature of the solution.
	}
	\label{fig:1-form-adaptive-lchain}
\end{figure}

\subsection{Maxwell eigenvalue problem}

Consider now the domain \(\domain \subseteq \mathbb{R}^2\) as the square \([0,\pi]^2\). In
this case, the problem in question is to find \(\omega \in \mathbb{R}\) and \(\vectrial \in
\Hcurlzero\) such that \(\vectrial \neq \boldvec{0}\) and
\[
	\langle \curl \vectrial, \curl \vectest \rangle = \omega^2\langle \vectrial, \vectest
	\rangle,\ \forall \vectest \in \Hcurlzero.
\]
It is well known that the eigenvalue solutions satisfy \(\omega^2 = m^2+n^2\) with \(m,n \in
\{0,1,\dots\}\). Once again, no harmonic vector fields are expected due to the topology of
the domain. We will also solve this problem on two distinct meshes, see
\Cref{fig:maxwell-example}, to highlight how spurious harmonic vector fields can alter the
solution in subtler ways than before.
\begin{figure}[htbp]
	\centering
	\hfill
	\begin{subfigure}[b]{0.4\textwidth}
		\centering

 		\caption{}
		\label{fig:maxwell-example-a}
	\end{subfigure}
	\hspace{10pt}
	\begin{subfigure}[b]{0.4\textwidth}
		\centering
		%
 		\caption{}
		\label{fig:maxwell-example-b}
	\end{subfigure}
	\hfill
	\strut

	\caption{
		Two hierarchical meshes for the test of Maxwell eigenvalue problem.
		Both meshes were created from the same set of marked elements, the
		difference being the use of \cref{alg:exact-mesh} in \Cref{fig:maxwell-example-b}.
	}
	\label{fig:maxwell-example}
\end{figure}
After discretizing the relevant weak form, we wish to find \(\omega_h \in \mathbb{R}\) and
\(\vectrial_h \in \hbspace^{1}_{L}, \vectrial_h \not = \boldvec{0},\) such that
\[
	\langle \curl \vectrial_h, \curl \vectest_h \rangle = \omega^2_h \langle \vectrial_h,
	\vectest_h \rangle,\quad \forall \vectest_h \in \hbspace^{1}_{L},
\]
with \(p=4\). In \cref{tab:maxwell-example} we show the first ten non-zero computed eigenvalues in
the two cases of \Cref{fig:maxwell-example}, showing that the lack of exactness causes the
appearance of four extraneous zero eigenvalues. Note that in the case of mesh
\Cref{fig:maxwell-example-a} the null-space has four extraneous null eigenvalues, because
the mesh has four problematic intersections. This is to be expected since each problematic
intersection is responsible for introducing harmonic forms, which belong to
\(\kernel(\curl)\).
The presence of these spurious modes is shown in \Cref{fig:maxwell-example-eigenvalues}. The
correction of the refinement, as performed by our algorithms, recovers the exactness of the
sequence and the correct spectrum.

\begin{table}[htbp]
	\centering
	\begin{tabular}{
			c
			S[table-format=1.2e-2, round-mode = places, round-precision = 2]
			S[table-format=1.2e-2, round-mode = places, round-precision = 2]
		}
		\toprule
		{}                                   & {\Cref{fig:maxwell-example-a}}  & {\Cref{fig:maxwell-example-b}}
		\\ \midrule
		\makecell{Null-space \\ dimension}                 & ~{4}                            & ~{0}
		\\ \midrule
		{Eigenvalue}                         & \multicolumn{2}{c}{Abs. ~error}
		\\ \cmidrule{2-3}
		\begin{filecontents*}{figures/maxwell-eigenvalue.csv}
eigenvalue,computed_a,computed_b,error_a,error_b
1.0,1.000000000002566,1.0000000000025622,2.5659474545136618e-12,2.5621726962299363e-12
1.0,1.0000000000025693,1.000000000002568,2.5692781235875373e-12,2.567945855957987e-12
2.0,2.0000000000062617,2.000000000006264,6.261657858885883e-12,6.263878304935133e-12
4.0,4.000000002358036,4.000000002194848,2.3580364327813186e-9,2.1948478590161358e-9
4.0,4.000000002361948,4.0000000021962165,2.361947970541678e-9,2.1962165419608937e-9
5.0,5.000000002828646,5.000000002928438,2.8286457620652072e-9,2.9284379365890345e-9
5.0,5.000000002828718,5.000000002928462,2.828717704517203e-9,2.9284619174063664e-9
8.0,8.000000006012925,8.000000004833007,6.012925268805702e-9,4.833006883586677e-9
9.0,9.000000160281875,9.000000160332105,1.6028187488359436e-7,1.6033210492594208e-7
9.0,9.000000160281951,9.000000160332231,1.6028195126693845e-7,1.6033223104727767e-7
10.0,10.000000211208642,10.000000211248599,2.1120864168722164e-7,2.112485990579671e-7
10.0,10.000000213303485,10.0000002134034,2.1330348509707164e-7,2.134033998402174e-7
13.0,13.000000120890336,13.000000122224682,1.2089033596396348e-7,1.2222468193101577e-7
13.0,13.00000012089044,13.000000122224863,1.20890440769017e-7,1.2222486311941338e-7
16.0,16.00000317379015,16.000003010560725,3.1737901515782596e-6,3.0105607251584843e-6
16.0,16.000003263995698,16.000003092987846,3.2639956977220663e-6,3.0929878462870875e-6
17.0,17.00000426638846,17.00000429232687,4.266388458518122e-6,4.292326870825036e-6
17.0,17.00000426638852,17.00000429232694,4.266388518914255e-6,4.292326938326596e-6
18.0,18.000000367158226,18.000000368059613,3.671582256004058e-7,3.680596130095637e-7
20.0,20.00000292767552,20.000002544491995,2.9276755206808502e-6,2.5444919948824918e-6
20.0,20.00000298056841,20.000002595060824,2.980568410748674e-6,2.595060824006623e-6
25.0,25.00000317591394,25.00000350404488,3.175913938946451e-6,3.5040448800316426e-6
25.0,25.000003175913946,25.000003504044923,3.175913946051878e-6,3.5040449226642068e-6
25.0,25.00003310400432,25.000035150449673,3.310400431999483e-5,3.515044967272729e-5
25.0,25.00003310400436,25.000035150449797,3.310400435907468e-5,3.515044979707227e-5
26.0,26.000049072404504,26.00004907958379,4.907240450435779e-5,4.907958378908006e-5
26.0,26.000050318952113,26.000050417293835,5.031895211260462e-5,5.041729383492566e-5
29.0,27.219465645587455,29.000031125856147,1.780534354412545,3.11258561467298e-5
29.0,27.219465645589185,29.000031125856243,1.7805343544108148,3.112585624265307e-5
32.0,29.000033639595387,32.000004348364044,2.9999663604046134,4.3483640439490046e-6
34.0,29.000033639595507,34.000037692736484,4.999966360404493,3.7692736484018496e-5
34.0,30.37509671434976,34.000038301099075,3.624903285650241,3.830109907454471e-5
36.0,32.00000632779199,36.00030592857204,3.9999936722080065,0.00030592857203970425
36.0,34.00003755984177,36.00031219424188,1.999962440158228,0.0003121942418786716
37.0,34.000037671541435,37.00041568758857,2.9999623284585653,0.0004156875885712452
37.0,36.00031458327956,37.00041568758861,0.999685416720439,0.00041568758860677235
40.0,36.00032422832829,40.00026560900052,3.9996757716717113,0.0002656090005217493
40.0,37.00041593333664,40.000268685448155,2.9995840666633597,0.0002686854481552814
41.0,37.00041593333685,41.000027663832675,3.9995840666631466,2.766383267527317e-5
41.0,40.0002865598348,41.00002766383278,0.9997134401652019,2.766383278185458e-5
45.0,40.00028669981107,45.000339499382115,4.99971330018893,0.0003394993821146386
45.0,41.00002755106934,45.000339499382285,3.999972448930663,0.00033949938228516885
49.0,41.000027551069714,49.00178973759598,7.999972448930286,0.0017897375959776696
49.0,42.77191663421235,49.00178973759606,6.228083365787647,0.0017897375960629347
50.0,45.00033897822371,50.000041116485896,4.999661021776291,4.111648589599781e-5
50.0,45.00033897822379,50.00241921293674,4.9996610217762125,0.00241921293674352
50.0,49.001798941805745,50.00244371943173,0.9982010581942546,0.0024437194317314948
52.0,49.00179894180617,52.000272044900555,2.9982010581938283,0.0002720449005551018
52.0,50.00004393626501,52.00028355639755,1.9999560637349916,0.00028355639754806816
53.0,50.0024202813236,53.00171903191381,2.9975797186763984,0.0017190319138080667
\end{filecontents*}
\csvreader[
		late after line =                                                                                       \\,
			filter ifthen = { \thecsvrow < 10 },
		]{figures/maxwell-eigenvalue.csv}{1=\eigenvalue, 4=\errorA, 5=\errorB}{\num[drop-zero-decimal]{\eigenvalue} & \errorA                         & \errorB
		}\bottomrule
	\end{tabular}
	\caption{
		First 10 exact eigenvalues of the Maxwell eigenvalue problem and the absolute
		errors for the computed values using the meshes of \Cref{fig:maxwell-example}.
	}
	\label{tab:maxwell-example}
\end{table}

\begin{figure}[ht]
	\centering
	\begin{tikzpicture}
		\begin{axis}[
				xlabel={Non-null spectrum rank},
				ylabel={Eigenvalue},
				legend entries={
						Exact, \Cref{fig:maxwell-example-a}, \Cref{fig:maxwell-example-b}
					},
				legend pos=north west,
				width=\textwidth,
				height=8cm,
			]
			\addplot+[
				only marks, mark size=3pt, mark=o, color=black,
			] table [
					x expr=\coordindex+1, y=eigenvalue, col sep=comma
				] {figures/maxwell-eigenvalue.csv};
			\addplot+[
				only marks, mark size=3pt, mark=+, mark options={draw=red, fill=black}
			] table [
					x expr=\coordindex+1, y=computed_a, col sep=comma
				] {figures/maxwell-eigenvalue.csv};
			\addplot+[
				only marks, mark size=3pt, mark=x, mark options={draw=black, fill=black}
			] table [
					x expr=\coordindex+1, y=computed_b, col sep=comma
				] {figures/maxwell-eigenvalue.csv};
		\end{axis}
	\end{tikzpicture}
	\caption{
		Comparison of the first 50 non-null eigenvalues computed using the two meshes of
		\Cref{fig:maxwell-example}. We are now accounting for both the offset of the
		gradients of scalar B-splines and the harmonic forms introduced by problematic
		intersections, which is 4 and 0 for the cases \Cref{fig:maxwell-example-a} and
		\Cref{fig:maxwell-example-b}, respectively. Due to the four problematic
		intersections, the refinement configuration introduces four spurious eigenvalues,
		that in this case are around 27.2, 27.2, 30.4, and 42.8. This behaviour is of course
		not present when problematic intersections are absent.
	}
	\label{fig:maxwell-example-eigenvalues}
\end{figure}

 \section{Conclusions}\label{sec:conclusion}

A key ingredient in the construction of stable numerical methods for electromagnetics and
fluid mechanics is the construction of a discrete de Rham complex that preserves the
cohomology of the continuous one. While the construction of such complexes with
tensor-product splines is well-understood, doing so with splines that allow for local
refinement is not. In this paper, focusing on the use of hierarchical B-spline spaces in two
dimensions, we have presented a first, constructive approach for building an
adaptively-refinable discrete de Rham complex with the right cohomology.

Our approach builds upon recent results from \cite{Shepherd2024}. They describe sufficient
conditions for hierarchical meshes (in an arbitrary number of dimensions) such that the
hierarchical B-spline spaces built on them form an exact discrete de Rham complex. In two
dimensions, given a hierarchical mesh which does not satisfy those conditions, we have now
shown how it can be modified to do so, by additional refinement of L-shaped chains. We have
also presented algorithms that show how our theoretical results can be implemented.
Moreover, we prove that, if the first space in the de Rham complex is of $\mathcal{H}$- or
$\mathcal{T}$-admissible of class $m$, then so are the rest of the spaces of the complex.
Ensuring admissibility is important not just from the point of view of the theory of
hierarchical B-splines \cite{Buffa2015}, but also for the stability of the numerical methods
built on them \cite{Carraturo2019}. Our theoretical results imply that the algorithms
developed for maintaining admissibility of hierarchical B-spline spaces (e.g.,
\cite{Bracco2019}) can be immediately used in the context of spline differential forms.
Finally, we present numerical tests (both source and eigenvalue problems) that demonstrate
the effectiveness of our approach.

There are several interesting lines of research that remain open. For instance, in
higher-dimensional domains a different mesh adaption is needed. The main algorithm proposed
in this paper relies entirely on the assumption that a limited number of pair-wise checks
are sufficient to confirm, or rule out, exactness of the hierarchical complex at each level;
this is not the case for higher-dimensional domains. It is therefore likely that an
algorithm akin to the one we propose could, for higher dimensions, introduce an unpractical
amount of computational complexity. In that scenario, to ensure exactness of the
hierarchical complex it is preferable to a priori restrict the allowed refinement domains
such that exactness is always guaranteed, such as proposed in \cite{Dijkstra2026}.

It will also be appealing to study if the smallest quantum of refinement could be
reduced from supports of $0$-form B-splines to supports of volume-form B-splines, as was
done in \cite{Evans2018}, thereby yielding slightly more flexible refinements. Another
worthwhile research direction was alluded to at the end of Section \ref{sec:preliminaries}:
the formulation of bounded cochain-projections for the hierarchical de Rham complex. These
and other topics will be the subject of future work.
 
\section*{Acknowledgments}

Diogo C. Cabanas is supported by FCT - Fundação para a Ciência e Tecnologia, I.P., with
project reference 2023.00238.BD and DOI identifier https://doi.org/10.54499/2023.00238.BD. 
The research of Deepesh Toshniwal is supported by project number 212.150 awarded through the
Veni research programme by the Dutch Research Council (NWO).
 
\bibliographystyle{amsplain}
\bibliography{bibliography}
\end{document}